\documentclass{siamart}

\usepackage{amsfonts}
\usepackage{amssymb}
\usepackage{amsmath} 
\usepackage{ntheorem}
\newcommand{\A}{\mathcal{A}}

\newcommand{\calI}{\mathcal{I}}

\newcommand{\calA}{\mathcal{A}}

\newcommand{\calF}{\mathcal{F}}
\newcommand{\calK}{\mathcal{K}}
\newcommand{\calS}{\mathcal{S}}

\newcommand{\calQ}{\mathcal{Q}}
\newcommand{\calL}{\mathcal{L}}

\newcommand{\calE}{\mathcal{E}}

\newcommand{\tr}{\mathrm{Tr}}

\newcommand{\norm}[1]{\| #1 \|}
\newcommand{\inner}[2]{\langle #1, #2 \rangle}

\newcommand{\C}{\mathbb{C}}

\newcommand{\Span}{\mathrm{span}}

\newcommand{\ext}{\mathrm{ext}}
\newcommand{\rmax}{r_{\max}}

\newcommand{\la}{\langle}
\newcommand{\ra}{\rangle}

\newcommand{\R}{\mathbb{R}}
\def\be{\begin{equation}}
\def\ee{\end{equation}}
\def\bi{\begin{itemize}}
\def\ei{\end{itemize}}
\newcommand{\X}{\mathcal{X}}

\newcommand{\Herm}{\mathcal{S}}

\newcommand{\gram}{\mathrm{Gram}}
\newcommand{\real}{\mathbb{R}}

\newcommand{\rank}{\mathrm{rank}}

\newcommand{\CPSD}{\mathcal{CS}_+}
\newcommand{\CP}{\mathcal{CP}}
\newcommand{\DNN}{\mathcal{DNN}}

\newcommand{\sfT}{{\sf T}}

\newcommand{\CS}{\mathcal{CS}_+}
\newcommand{\CPSDR}{\mathrm{cpsd}\textnormal{-rank}}
\newcommand{\CL}{\mathcal{GL}}
\newcommand{\calh}{\mathcal{H}}
\newcommand{\call}{\mathcal{L}}
\newcommand{\pbf}{\mathbf{p}}
\newcommand{\cbf}{\mathbf{c}}
\newcommand{\cornm}{\mathrm{Cor}(n,m)}
\newcommand{\comment}[1]{}

\usepackage{color,graphicx}

\newtheorem{remark}{Remark}[section]
\newtheorem{example}{Example}[section]
\newtheorem{thm}{Result}

\newtheorem*{thmnewww}{Theorem}

\newcommand{\Tr}{\mathrm{Tr}}

\newcommand{\pabxy}{(p(ab|xy))}
\def\be{\begin{equation}}
\def\ee{\end{equation}}
\newcommand{\PSDR}{\mathrm{rank}_{psd}}
\newcommand{\CPR}{\mathrm{cp}\textnormal{-rank}}

\title{Completely positive semidefinite rank}
\author{
A. Prakash
\thanks{Centre for Quantum Technologies, National University of Singapore and Nanyang Technological University, Singapore.
\email{aprakash@ntu.edu.sg}}.
\and
J. Sikora
\thanks{Centre for Quantum Technologies, National University of Singapore,
\email{cqtjwjs@nus.edu.sg}}.
\and
A. Varvitsiotis
\thanks{Centre for Quantum Technologies, National University of Singapore and Nanyang Technological University, Singapore,
\email{avarvits@gmail.com}}.
\and
Z. Wei
\thanks{Centre for Quantum Technologies, National University of Singapore and  Nanyang Technological University, Singapore, \email{weizhaohui@gmail.com}}.
}

\begin{document}

\maketitle

\begin{abstract}
{An $n\times n$} matrix $X$ is called completely positive
semidefinite (cpsd) if there  exist $d\times d$ Hermitian positive
semidefinite {matrices} $\{P_i\}_{i=1}^n$ (for some $d\ge 1$) such
that  $X_{ij}= \tr(P_iP_j),$ for all {$i,j \in \{ 1, \ldots, n \}$}.
The $\CPSDR$ of a cpsd matrix is the smallest $d\ge 1$ for which
such a representation is possible. In this work  we initiate the
study of the $\CPSDR$ which we motivate twofold. First, the $\CPSDR$
is a natural non-commutative analogue of the 
{completely positive rank} of a completely positive matrix.
{Second}, we show that the $\CPSDR$ is physically motivated as it
can  be used to upper and lower bound the size of a quantum system
needed to generate a quantum~behavior.

In this work  we {present}  several  properties of the cpsd-rank.
Unlike the completely positive rank  which is at most  quadratic in the size of the matrix, no general upper bound is  known on the cpsd-rank of a cpsd matrix.  In fact, we show that the cpsd-rank can be exponential in terms of the size. Specifically,   for any $n\ge1,$ we construct  a  cpsd matrix of size $2n$ whose  cpsd-rank  is $2^{\Omega(\sqrt{n})}$. Our construction is based on  Gram  matrices of Lorentz  cone vectors, which we show are cpsd. The proof  relies crucially on   the connection between the cpsd-rank  and quantum behaviors. In particular, we use  a  {known} lower bound on the
{size of matrix representations of}
extremal quantum correlations
{which we apply to}
high-rank extreme points of the {$n$-dimensional} elliptope.

{Lastly, we study cpsd-graphs, i.e., graphs $G$ with the property that every doubly nonnegative matrix whose support is given by $G$ is cpsd. We show that a graph  is cpsd  if and only if  it has no odd cycle of length at least {$5$}  as a subgraph.  This coincides with  the characterization of cp-graphs.}
\end{abstract}


%\begin{AMS}
%\end{AMS}

\section{Introduction}

\subsection{Setting the scene}
Consider a family of vectors $\{v_i\}_{i=1}^n$  such that the angle between any pair of them
is at most $\pi/2$.
A necessary and sufficient condition for showing that the
configuration  $\{v_i\}_{i=1}^n$  admits an  isometry  to {some}
nonnegative orthant  is that the $n\times n$ matrix $(\la v_i,v_j\ra_{1\le i,j\le
n}),$  formed by collecting all  pairwise inner products of the
vectors $\{v_i\}_{i=1}^n,$ is \emph{completely positive}. Formally, a
symmetric $n\times  n $ matrix $X$ is called {\em completely
positive} (cp)  if there exist vectors $\{p_i\}_{i=1}^n\subseteq \R^d_+,$
for some $d\ge 1,$ such that $X_{ij}=\la
p_i,p_j\ra,$ for all $1\leq i,j \leq n$.

The set of  $n\times n$ completely positive matrices, denoted  by $\CP^n$, forms a full-dimensional, pointed, closed convex cone whose structure   has been   extensively studied  (e.g. see~\cite{CP}). Linear conic programming over the $\CP$ cone is  particularly interesting due to its  expressive power. Specifically, any nonconvex quadratic program having  both  binary and continuous variables can be cast as a linear  conic  program over the $\CP$ cone \cite{Bur07}. In particular, this implies that  optimization over the $\CP$ cone   is intractable. On the positive side, there exist inner \cite{Las12} and outer   \cite{Par00} semidefinite programming hierarchies that can be used to  approximate  the $\CP$ cone.

In this work we focus on  a generalization  of the embeddability
question considered  above: When can  a family of vectors
$\{v_i\}_{i=1}^n $ whose pairwise inner products are
nonnegative be isometrically embedded into a
cone of Hermitian positive semidefinite  matrices?  Throughout, {we}
denote by $\calh^d_+$ the cone  of $d\times d$  Hermitian positive
semidefinite (psd) matrices {and by $\calS_+^d$ the set of $d\times d$ symmetric psd matrices}.  Formally, we are asking for the
existence of matrices  $\{P_i\}_{i=1}^n\subseteq
{\calh^{d}_+}$, for some  ${d\ge1}$, satisfying
$$\la
v_i,v_j\ra=\tr(P_iP_j), \text{ for all }1 \leq i,j\leq n.$$

Since the direct sum of two psd matrices is again psd,  the set of $n\times n$  matrices of the form $( \tr(P_iP_i)_{1\le i,j\le n}), $
where  $\{P_i\}_{i=1}^n\subseteq \calh^d_+$ (for some $d\ge 1)$,
forms a convex cone. This {set of matrices is} denoted by $\CPSD^n$ and is known  as
the cone of  {\em completely positive semidefinite {(cpsd)}
matrices}.

The $\CPSD^n$ cone was introduced recently to provide linear conic formulations for various quantum  graph parameters~\cite{LP14,R14b}.  Subsequently, it was shown  in \cite{SV} that underlying these formulations is the fact that the set of quantum behaviors   can be expressed as the projection of an affine section  of the $\CPSD^n$ cone (cf. Theorem \ref{thm:conicformulation}).

Clearly,  for  every $n\ge 1$ we have that
$  \CP^n \subseteq \CPSD^n \subseteq \DNN^n$, where we denote by $\DNN^n$ the set of $n\times n$ {\em doubly nonnegative} matrices, i.e., matrices that are  positive semidefinite and  entrywise nonnegative. For the rightmost inclusion recall that   the trace inner product of two psd matrices is a nonnegative scalar.  The leftmost inclusion holds since nonnegative vectors correspond to diagonal psd matrices.

It is known that   $\CP^n= \DNN^n$ for   $n\le 4$   \cite{MM61}, whereas  for $n\ge 5,$ all
inclusions  given above are known to be strict. In particular, it follows from \cite{FW} that $\CP^6\ne \CPSD^6$  and by \cite{FGPRT} that  $\CP^5\ne \CPSD^5$.  Furthermore,
it was  shown in  \cite{FW}  that $\CPSD^5\ne \DNN^5$ and in \cite{LP14}  that ${\rm cl}(\CPSD^5) {\subsetneq}  \DNN^5$, where ${\rm cl}(\CPSD^n)$ denotes the closure of $\CPSD^n$. Lastly, it was shown in \cite{LP14} that for any matrix $X$ whose support is {a cycle}  we have that $X\in \CP$ if and only if $X\in \CPSD$. Furthermore, it is known that for every odd cycle $C_{2t+1} \ (t\ge 2)$ there exists a  matrix in   $\DNN\setminus \CP$ whose support is given by $C_{2t+1}$ (see \cite[Theorem~2.12]{CP}).  Combined with the above, this fact gives a family of matrices in $\DNN\setminus \CPSD$  that are supported by $C_{2t+1},$ for all $t\ge 2$.

Not many things are  known  concerning  the  structure of   {$\CPSD^n$}. In
particular it is not known whether {$\CPSD^n$} is closed. The
closure of $\CPSD^n$  was  characterized  in  \cite{BLP} as
the set of doubly nonnegative matrices that admit a Gram
factorization using positive elements in a certain finite  von
Neumann algebra,  an infinite dimensional analogue of $\CPSD$-factorizations (cf. Section \ref{thm:neccondition1}). Furthermore,
combining results from \cite{LP14} and~\cite{Ji13}   it follows that linear
optimization over {$\CPSD^n$} is~NP-hard.

Given a  completely positive matrix $X\in \CP^n$, the smallest integer $d\ge1$ for which there exist vectors $\{p_i\}_{i=1}^n\subseteq \R^d_+$  satisfying  $X_{ij}=\la p_i,p_j\ra,$ for all $1 \leq i,j \leq n$ is called the {{\em completely positive rank (cp-rank)}  of $X$, and is  denoted by $\CPR(X)$.

 A  very useful property of the $\CPR$ is that it admits an atomic reformulation.
Specifically,   the  $\CPR(X)$  of a matrix   $X\in \CP^n$ can be equivalently defined as  the smallest  $d\ge 1$ for which there exist   vectors $\{x_i\}_{i=1}^d\subseteq \R^n_+$ satisfying~$X=\sum_{i=1}^d  x_ix_i^\sfT$.

 Studying the properties of the  $\CPR$   is a problem that has received significant  attention.
 By the conic analogue  of Carath\'eodory's
 Theorem  (e.g. see  \cite[Theorem 1.34]{CP}) and the atomic reformulation of the $\CPR$   described above it follows that  for any   $X\in\CP^n$   we have   $\CPR(X)\le \binom{n+1}{2}$.
  At present, the  best upper bound is     $ {n^2\over 2}+O(n^{3/2}),$  for any $X\in \CP^n$
 \cite{BSU}. Moreover, this upper bound  is asymptotically  tight with respect to the  Drew-Johnson-Loewy  lower bound of $\big\lfloor {n^2\over 4}\big\rfloor,$ for $n\ge 4$ \cite{DJL}.

The definition of the $\CPSD$ cone suggests the  following generalization of the notion of $\CPR$, where nonnegative vectors are replaced by Hermitian psd matrices.
\medskip 

\begin{definition}The {\em completely positive semidefinite rank} {($\CPSDR$)} of a matrix $X\in \CPSD^n$, denoted by  $\CPSDR(X)$,  is defined as the {least} $d\ge 1$ for which there exist matrices  $\{P_i\}_{i=1}^n\subseteq\calh^d_+$  such that  $X_{ij}=\tr(P_iP_j), $\  for all $i,j\in [n]$.
\end{definition}

Given a matrix $X\in \CPSD^n$, we refer to  any family of matrices
$\{P_i\}_{i=1}^n\subseteq\calh^d_+$ such that  $X_{ij}=\tr(P_iP_j),
$\  for all $i,j\in [n]$, as a {\em $\CPSD$-factorization} of $X$.
Furthermore, we  call a  $\CPSD$-factorization  {\em size-optimal}
if  the size of {each} $P_i$ is equal to  $\CPSDR(X)$.

The notion of the cpsd-rank was introduced recently in \cite{FGPRT}
(as a   variant of the psd-rank) although its properties
were not studied there. Our goal in  this work is to   initiate the
study  of the %completely positive semidefinite rank
{cpsd-rank} of a cpsd matrix.

The study of the $\CPSDR$  is motivated as follows. First, the $\CPSDR$ is  a
natural non-commutative  generalization of the well-studied notion
of $\CPR$. {Second},  and most {important}, we show  that  the
 $\CPSDR$ enjoys strong physical motivation. Specifically, we
show that some fundamental questions concerning the $\CPSDR$ are
intimately  related to long standing   open problems {on} the foundations of
quantum mechanics. This is explained in detail in the
following~section.

\subsection{Physical  motivation}\label{sec:quantumcorrelations}

A  {\em Bell scenario} is a {physical} experiment involving two
{spatially separated}
{parties}, Alice and Bob, who perform {local} measurements on a
shared physical system. For our purposes,  imagine that
Alice and Bob are individually given a  closed  box, whose inner
{workings} are unknown to both parties. The boxes work as follows:
Alice's (resp. Bob's) box has
{$m_A$ (resp. $m_B$) different buttons. After  each party presses  a button, the box displays
one out of $o_A$ (resp. $o_B$) possible~outcomes}.

It is instructive to think of the boxes as measurement devices and the content  of the boxes as a physical system that  each party has in his possession. Furthermore, each button
corresponds to a
choice of measurement that can be performed on the system
and the {displayed outcome}
corresponds to  the outcome of the measurement.

{The object of interest in a Bell scenario are {the} statistics that can be obtained via such a pair of boxes. Specifically,
suppose that Alice and Bob synchronize their clocks and  distance themselves from each other so that they
cannot communicate. After they are sufficiently far apart, they simultaneously press a button on their box
(chosen randomly and independently) and record the button that they
pressed and the displayed outcome. After repeating the whole process a sufficient number of times\footnote{To be precise, each time this is repeated each party  should receive a new  copy of the box.}
Alice and Bob
meet to calculate the joint conditional probabilities $p(ab|xy)$, i.e.,
the probability that upon pressing buttons
$x\in[m_A]$, $y\in [m_B]$, they  obtained the outcomes $a\in [o_A]$ and $b\in [o_B]$, respectively.
These probabilities   are  arranged in a vector  $\pbf=\pabxy$ of length $m_Am_Bo_Ao_B$ which we call a~{\em behavior}.

{Suppose} that after the parties
compare their   statistics  they note that for some $x,y,a,b$ it is the case that
$p_A(a|x) p_B(b|y) \ne p(ab|xy)$, where $p_A(a|x)$ and
$p_B(b|y)$ denote the local marginal distributions of Alice and Bob,
respectively. This indicates  that the  outcomes  of the boxes are
statistically dependent.

A {\em local hidden variable}  (LHV)  model would  account for  this dependence by asserting    that  the two systems have interacted at some point in the past, and as a result they both depend  on some  ``hidden'' variable $k$. Once the  value of $k$ is taken into account,  then the probabilities decouple, i.e., $p(a|x,k)p(b|y,k)= p(ab|xyk)$.
Formally, we say that a behavior $\pbf=\pabxy$ admits a LHV model  ({also referred to as being local}) if there exist $ k_i\ge 0, \ m_{a}^{s,i}\ge 0,\  n_{b}^{t,i}\ge 0$ satisfying  $ \sum_i k_i = 1,$  and
  $\sum_a  m_{a}^{x,i} =\sum_b n_{b}^{y,i} = 1$ for all $x,y,i,$ such that   $ p(ab|xy)=\sum_{i} k_im_a^{x,i}n_b^{y,i},$ for all~$a,b,x,y$.

In this work we focus on the  case where  Alice and Bob
share a
 {\em quantum mechanical system} (e.g. each box   contains one of a pair of entangled
particles).
In this setting   the system is governed by the laws of quantum mechanics. In particular, the  outcome statistics  can be calculated  using  the mathematical formalism of quantum mechanics which we sketch below (see also Section~\ref{sec:preliminaries}).

According to the postulates of quantum mechanics, the state of the  quantum system shared by  Alice and Bob corresponds to  a   Hermitian psd matrix $\rho$ acting on $\C^d\otimes \C^d,$ with trace equal to $1$. The  measurement process is   described  by  two families of  Hermitian psd {operators}  $\{ M_{a|x}\}_a$  and $\{ N_{b|y}\}_b$, {each} acting on $\C^d,$ such that $\sum_a M_{a|x}=\sum_bN_{b|y}=I_d,$ for all $x,y$.  We say the behavior  $\pbf=\pabxy$ is {\em quantum} if there exists a quantum state $\rho$ and measurement operators $\{ M_{a|x}\}_a$  and $\{ N_{b|y}\}_b$ such that
$p(ab|xy) = \Tr \left( (M_{a|x} \otimes N_{b|y} ) \rho\right)$, for all $a,b,x,y$.

In this work, we refer to a Bell scenario with $m_A,m_B$ measurement settings and $o_A,o_B$ measurement outcomes as {an} $(m_A,m_B,o_A,o_B)$-scenario.
Furthermore, we  denote by $\calQ$ (resp. $\calL$) the set of  quantum (resp. local) behavior. To stress the dependence on the underlying Bell scenario  we use the notation $\calQ(m_A,m_B,o_A,o_B).$

Clearly, {$\calL\subseteq \calQ$} and it  is one of the pillars of
quantum information theory that there exist behaviors  that are
quantum but do not admit a local hidden variable  explanation, i.e.,
$ \calQ$ is a strict superset of $\calL$ \cite{B64,B66}. For an  overview on Bell
scenarios and the   properties of   quantum behaviors  the
 reader is referred to \cite{Brunner14}.

Any  quantum behavior $\pabxy$ is {\em no-signaling}, i.e., each
party's  local marginal distribution is independent of the other
party's choice of measurement. Algebraically, this is expressed {as}
$\sum_bp(ab|xy)= \sum_bp(ab|xy')$ for all $ y\ne y'$, and
symmetrically that $\sum_ap(ab|xy)= \sum_ap(ab|x'y)$ for all $x\ne
x'$. This implies that the local marginal distributions $(p_A(a|x))$
and $(p_B(b|y))$ are well-defined. In {a} Bell scenario where all the measurements
 have binary outcomes, we call a behavior $\pabxy$ {\em
unbiased} if
$p_A(a|x)=p_B(b|y)=1/2, \text{ for all } a,b,x,y. $

Given a quantum behavior $\pbf$, we refer to any ensemble $\{\rho,\{
M_{a|x}\}_a, \{ N_{b|y}\}_b\}$ such that $p(ab|xy) = \Tr\left((M_{a|x}
\otimes N_{b|y} )\rho\right)$ for all $a,b,x,y$ as a {\em quantum
representation} of~$\pbf$. A quantum behavior
$\pbf=\pabxy$ admits multiple quantum representations.  We say that
$\pbf\in \calQ$ admits a \emph{$d$-dimensional (quantum)
representation} if there exists a quantum representation $\{\rho,\{
M_{a|x}\}_a, \{ N_{b|y}\}_b\}$, where $\rho$ acts on $\C^d\otimes
\C^d$ {and $\{ M_{a|x}\}_a$ and $\{ N_{b|y}\}_b$ each act on} $\C^d$.
We denote  by $\mathcal{D}(\pbf)$ the smallest integer $d\ge 1$ for which the quantum behavior $\pbf=\pabxy$ admits a $d$-dimensional representation.

The starting point for our work  is a recent result from  \cite{SV} which states {that}  the set of quantum behaviors (resp. local)  can be expressed as a projection of an affine section of the completely positive semidefinite cone (resp. completely positive cone).

\medskip 
\begin{theorem}[\cite{SV}]\label{thm:conicformulation}
 Consider a behavior    $\pbf=\pabxy$ and set $n:=m_A o_A+m_B o_B$. The {behavior}   $\pbf$ is quantum (resp. local)  if and only if there exists  a matrix $R\in \CPSD^n$ (resp. {$\CP^n$})  indexed by
{$([m_A] \times [o_A]) \cup ([m_B] \times [o_B])$}
 such that
\begin{align}
 & {\sum_{a = 1}^{o_A} \sum_{a' = 1}^{o_A}} R_{xa,x'a'}=1, \text{ for all } x,x' {\in [m_A]}; \label{eq:cond1}\\
 & {\sum_{a = 1}^{o_A} \sum_{b = 1}^{o_B}} R_{xa,yb}=1, \text{ for all } x {\in [m_A]}, \ y {\in [m_B]};\label{eq:cond2}\\
  & {\sum_{b=1}^{o_B} \sum_{b'=1}^{o_B}} R_{yb,y'b'}=1, \text{ for all } y,y' {\in [m_B]};\label{eq:cond3}\\
  & R_{xa,yb}=p(ab|xy), \text{ for all } {a \in [o_A], \ b \in [o_B], \ x \in [m_A], \ y \in [m_B]}. \label{eq:cond4}
\end{align}
\end{theorem}

For a fixed $\pbf$ we denote by $\mathcal{A}(\pbf) $ the  affine
subspace of $\calS^n$  {consisting} of matrices that satisfy
\eqref{eq:cond1}, \eqref{eq:cond2},\eqref{eq:cond3} and
\eqref{eq:cond4}, {where $\calS^n$ is the set of $n\times n$
symmetric matrices}.

{By combining  the linear conic formulations from Theorem
\ref{thm:conicformulation} with {a} reduction from  \cite{SVW15}
we have that $\mathcal{D}(\pbf)$ corresponds to}  the smallest size
of a $\CPSD$-factorization over all matrices {$R\in \CPSD^n\cap
\mathcal{A}(\pbf)$}. Using the notion of the cpsd-rank this can be
equivalently expressed as follows.

\medskip 
\begin{theorem}[\cite{SV, SVW15}]\label{thm:minsizequantumcorrelation}
For any $ \pbf=\pabxy\in \calQ$ we have that
\[ {\mathcal{D}(\pbf) = {\min} \left\{ \CPSDR(R): {R\in \CPSD^n\cap \mathcal{A}(\pbf)}\right\}}. \]
\end{theorem}
\medskip

For a fixed Bell scenario,  it is a fundamental problem  to understand whether  one can place  a {\em finite} upper bound on the size of the quantum system necessary to generate all quantum behaviors. In mathematical terms,  the question is to decide  whether
$ \max\{ \mathcal{D}(\pbf): \pbf\in Q\}$ is  finite or  infinite, {where, {again}, $Q$ denotes  the set of all quantum behaviors corresponding to  this fixed Bell scenario.

There is no  clear consensus  whether finite dimensions always
suffice.  It follows from the work of Tsirelson~\cite{TS87} that in a $(m_A,m_B,2,2)$-scenario finite dimensions are {sufficient} to generate all {\em unbiased} behaviors (cf. Section \ref{sec:tsirelson}).
Furthermore, in a
{$(1,1,o_A,o_B)$}-scenario (i.e., exactly one
measurement setting per party),  the sets of
local and quantum behaviors coincide  and additionally, it  was shown by Jain, Shi, Wei and Zhang \cite{JSWZ}
that  $\mathcal{D}(\pbf)$ is equal to the positive semidefinite rank
of the nonnegative matrix ${P=(p(ab)_{1 \leq a \leq o_A, 1 \leq b \leq o_B}) \in \R^{o_A\times o_B}_+}$  (cf. Section \ref{psdvscpsd}).   The latter quantity is upper
bounded (e.g. by ${\min\{o_A,o_B\})}, $ so again  in this case
{the maximum of} $\mathcal{D}(\pbf)$ over all behaviors is finite. On the other
hand, P\'al and V\'ertesi in \cite{Pal10} provide {numerical} evidence
that finite dimensional quantum systems do not suffice  in the
(3,3,2,2)-scenario, although this still remains to be proven analytically.

Our motivation for introducing and studying the $\CPSDR$  is that it
provides a novel  approach to address  the finite vs. infinite
representability {problem} of the set of quantum behaviors.
Specifically, using Theorem \ref{thm:minsizequantumcorrelation} we
immediately get  two sufficient conditions, in terms of the
$\CPSDR$,  that allow us to either prove or disprove that {finite-dimensional} systems  suffice to generate all
quantum~behaviors.

\medskip
\begin{proposition}\label{eq:proposition}
Fix a  $(m_A,m_B,o_A,o_B)$-scenario,  set $ n:=m_Ao_A+m_Bo_B$ and let $\calQ$ be the corresponding set of quantum behaviors. We have that:
\medskip
\bi
 \item[$(i)$] If  $\max\{ \CPSDR(X): X\in \CPSD^n \}<+\infty$ then  $ \max\{ \mathcal{D}(\pbf): {\pbf\in {\calQ}}\}<+\infty$;\medskip
 \item[$(ii)$]
Say that
 for every $d\ge 1$ there exists $p_d\in \calQ$ such that  for any
 $R\in \CPSD^n\cap \mathcal{A}(\pbf_d)$  we have   $\CPSDR(R_d)>d$. Then
 $ \max\{  \mathcal{D}(\pbf): {\pbf\in {\calQ}}\}=+\infty$.
 \ei
\end{proposition}
\medskip

The value of   Proposition \ref{eq:proposition} is that {it}
identifies a concrete   mathematical problem, stripped {off} all
quantum mechanical context,  whose resolution would  settle  the
question of finite vs. infinite  dimensionality of the set of
quantum behaviors:

\medskip

\begin{center}{\em Question:} Is $ \max\{ \CPSDR(X): X\in \CPSD^n \}$ finite or infinite?
\end{center}

\medskip

The question concerning  the finiteness of the $\CPSDR$ was already stated in~\cite{FGPRT}.
As already mentioned,  if we pose  the same question but  replace  {$\CPSD^n$}  by the cone of  completely positive matrices, the answer is {known}: The  $\CPR$  can be at most  quadratic in the size of the matrix. The proof of this fact relies  on  the atomic reformulation for the cp-rank. On the other hand,  we are not aware of such an atomic  {reformulation} for the $\CPSDR$ and this limits the analogies with the {$\CP^n$}~case.

\subsection{Contributions and paper organization}\label{sec:contributions}

In this work we initiate the systematic study of the  $\CPSDR$ and by establishing a connection to quantum behaviors, we make  the case  that it admits   significant physical motivation.

In Section \ref{sec:preliminaries} we introduce all necessary notation, definitions and background  material on Linear Algebra, Quantum Mechanics, Convexity and Graph Theory.
  We begin our study of the $\CPSDR$ in  Section \ref{sec:properties}  where our goal is to give a general introduction and   collect  basic properties.
Specifically, in Section \ref{sec:basic} we consider matrix operations that preserve  the property of being cpsd and examine how they affect the $\CPSDR$.  In Section~\ref{sec:lowerbounds}  we identify an analytic and support based lower bound  on the $\CPSDR$ and note that both the bounds
never exceed the size of the matrix. Lastly, in Section \ref{sec:comparison} we relate the $\CPSDR$ to other notions of matrix
ranks.

As  was already mentioned, no general  upper bound  is currently known on the $\CPSDR$ of a {cpsd matrix}.
In view of this,  there are two natural research directions: First, identify families of cpsd matrices for which we can place an upper   bound on the $\CPSDR$ and {second}, identify cpsd matrices with high $\CPSDR$.
 As we describe below, in this work we make   progress in both directions.

In Section \ref{sec:gramlorentz} we consider the  question of upper bounding the $\CPSDR$ for certain  families of $\CPSD$ matrices. {We focus on}
 {\em Gram-Lorentz} matrices,  denoted by $\CL$, defined as the set of Gram matrices of Lorentz cone vectors (also known as {the} second-order cone)  which we introduce and study in Section \ref{sec:embedding}. Furthermore, in Section~\ref{sec:embedding} we revisit  and give a simplified proof of  a construction from \cite{FW}, where it {is} shown that the $m$-dimensional  Lorentz cone can be isometrically embedded into a psd cone of size $2^{\Omega(m)}$. This  implies  that   Gram-Lorentz  matrices  are cpsd. Furthermore,  in Section~\ref{sec:glmatrices}   we show that for any $X\in \CL$ we have that $\CPSDR(X)\le 2^{O(\rank(X))}$.

As it turns out, Gram-Lorentz matrices are also useful  to construct  matrices   that are cpsd but not completely positive. The first such  separation $\CP^6 {\subsetneq} \CPSD^6$ was  in fact shown   using $\CL$ matrices~\cite{FW}.   In Section \ref{sec:cpsdcpseparation} we generalize the construction from \cite{FW} and identify a sufficient condition for constructing matrices in $\CPSD\setminus \CP$.

Lastly, Gram-Lorentz matrices are also relevant in the context of
quantum behaviors.  In view of Theorem \ref{thm:conicformulation}
given above,  any      $\calK\subseteq \CPSD$ corresponds to a
subset of the set of quantum behaviors. In Section~\ref{sec:unbiasedarelorentz} we
introduce and study {\em Gram-Lorentz  behaviors}, i.e., the quantum
behaviors that correspond to  $\calK=\CL$. Since $\CL$ matrices have
bounded cpsd-rank, all $\CL$ behaviors can be generated using  a
{finite-dimensional} quantum system. This  is again very interesting
since, as we mentioned in Section \ref{sec:quantumcorrelations}, it
is not known whether there exists a finite  upper  bound on the size
of a quantum system necessary  to generate  all quantum behaviors
corresponding to a  Bell scenario.

In Section \ref{sec:explowerbounds} we turn to the problem of constructing  cpsd matrices with high $\CPSDR$.
Interestingly,  Gram-Lorentz matrices turn out  to be the right tool  to address  this problem.
Indeed, {for our}  main result in    Section \ref{sec:explowerbounds} (cf. Theorem~\ref{thm:lboundgl}) we construct a family of $\CL$ matrices whose cpsd-rank is exponential in terms of their~size.

\medskip
\begin{thm}\label{res:first}
For  any {integer} $n\ge 1$ there  exists a matrix   $X_n\in \CL^{2n}$ such that
\be\label{cdwferfer}
 \CPSDR(X_n)\ge  {\sqrt{2}^{\lfloor \rmax(n)/ 2 \rfloor}}, \text{ where } {r_{\max}(n) := \left\lfloor {(\sqrt{1+8n}-1)/2}\right\rfloor}.
 \ee
  In particular, if  we take  $C_n$  to be  an extreme point of the $n$-dimensional elliptope $ \calE_n:=\{X\in \calS^n_+: X_{ii}=1,\  \forall i\in [n]\}$ satisfying   $\rank(C_n)=r_{\max}(n)$, then \eqref{cdwferfer} {holds for}
$$
X_n:= \begin{pmatrix}J+C_n & J-C_n\\J-C_n& J+C_n\end{pmatrix},
$$
where $J$ is the $n\times n$ matrix of all 1's.
\end{thm}

The starting point for proving  Result \ref{res:first} is  Theorem \ref{thm:minsizequantumcorrelation}. Specifically, it follows
by Theorem~\ref{thm:minsizequantumcorrelation} that  given a quantum behavior $\pbf \in \calQ,$  for any $  R\in \CPSD^n\cap \mathcal{A}(\pbf)$ we have that $\CPSDR(R)\ge\mathcal{D}(\pbf)$.
 Consequently, in order to derive   Result \ref{res:first}  it suffices to identify a sequence  of  Gram-Lorentz  behaviors $(\pbf_n)_{n\in \mathbb{N}}$  for which  all quantum representations require a quantum system of size exponential in~$n$.
We show that  for  any $n\ge 1$, there exists a
Gram-Lorentz behavior $\pbf_n$ corresponding to the $(n,n,2,2)$-scenario with the property that  $\mathcal{D}(\pbf_n)\ge  {\sqrt{2}^{\lfloor \rmax(n)/ 2 \rfloor}}$  (cf.   Theorem \ref{thm:GLbehaviorlowerbound}).  This is the {main step in} the proof of Result \ref{res:first}.

The first step towards constructing these Gram-Lorentz behaviors is  to  restrict to Bell scenarios where each party  has two possible outcomes, which we label by $\{\pm 1\}$. In this  case, instead of working with quantum  behaviors %$\pabxy\in\R^{4m_Am_B}$
we can equivalently work with the corresponding {\em correlation vectors}.  These are just the vectors that  correspond to the expected value  of the product of the player's individual outcomes.
This correspondence is  explained  in detail in Section~\ref{sec:tsirelson}.
Quantum correlation vectors   turn out to be extremely important  for this work  due to a  lower bound  on the size of  operator representations of extremal quantum correlations. This  result is implicit in~\cite{TS87} and  is explained in detail in Section \ref{sec:dsvefwefwe} in the Appendix.

In Section \ref{sec:glbehaviorsstuff} we construct  a family   Gram-Lorentz behaviors  $(\pbf_n)_{n\in \mathbb{N}}$ satisfying  $\mathcal{D}(\pbf_n)\ge  {\sqrt{2}^{\lfloor \rmax(n)/ 2 \rfloor}}$ (cf.   Theorem \ref{thm:GLbehaviorlowerbound}). To do this,
in Section~\ref{sec:corrtobeha},   we  translate the aforementioned lower bound  in  terms of  Gram-Lorentz behaviors.  Specifically, we show that to any extremal quantum correlation, {represented as a matrix $C$,} we can associate a Gram-Lorentz behavior $\pbf_C$  such that $\mathcal{D}(\pbf_C)\ge 2^{\Omega(\rank(C))}$.
  In view of this,
  %to conclude the proof of  Theorem \ref{thm:GLbehaviorlowerbound}
  it suffices to identify high-rank extremal quantum correlations. In Section~\ref{sec:extremepoints} we focus on the case $m_A=m_B=:n$ and show that the extreme points of the  $n$-dimensional elliptope $\calE_n$  are also  extreme points of the set of quantum correlations.
  This allows {us} to conclude the proof as it  is well-known that for any $n\ge1 $ there exist extreme points of $\calE_n$ whose rank is equal to $\rmax(n)$.
   In Section \ref{sec:puttingeverythingtogether} we put everything  together,
   and also  provide  an explicit family of Gram-Lorentz behaviors realizing this exponential lower bound. {Lastly, Section \ref{sec:highcpsdrank} is dedicated to the  proof of Result~\ref{res:first} where we  construct cpsd matrices with exponential  cpsd-rank.

In  Section \ref{sec:cpsdgraphs}   we study  cpsd-graphs, i.e., graphs $G$ with the property that every $\DNN$ matrix whose support is given by $G$ is also {in} $\CPSD$.   The analogous notion of cp-graphs has been  extensively studied (e.g. see \cite[Section~2.5]{CP}).
In particular, the class of cp-graphs admits an exact combinatorial characterization: A graph is cp  if and only if  it does not contain an odd cycle $C_{2t+1}$ ($t\ge 2$)  as a subgraph~\cite{KB}.

{We show} that the same characterization extends to  cpsd-graphs:
\medskip 
  \begin{thm}\label{res:third}
{A graph is cpsd  if and only if  it has no $C_{2t+1}$-subgraph $(t\ge 2)$.}
\end{thm}
\medskip

To prove Result \ref{res:third}, in Section \ref{thm:neccondition1} we generalize a construction from \cite{FW} and  \cite{LP14} and identify  a sufficient condition for constructing doubly-nonnegative matrices that do not admit a Gram factorization  using positive elements in any tracial von Neumann algebra. On the other hand,
 the closure of $\CPSD^n$  was  characterized  in  \cite[Theorem 4.6]{BLP} as
the set of psd matrices that admit a Gram
factorization using positive elements in a certain tracial   von
Neumann algebra.
  Thus,  our sufficient condition can be used to  construct matrices in $\DNN\setminus {\rm cl}  (\CPSD)$. Using these matrices,
 in Section~\ref{sec:whatever} (cf. Theorem~\ref{thm:cpsdgraphs}) we give the proof of Result \ref{res:third}.
 
\vspace{1.1cm}
\section{Preliminaries}\label{sec:preliminaries}

\subsection*{Linear Algebra}

We denote by $[d]$ the set $\{1,\dotsc,d\}$.
The standard  orthonormal basis of $\C^d$ is denoted by $\{e_i\}_{i=1}^d$, which we consider as column vectors.
The linear span of the vectors~$\{x_i\}_{i=1}^n$ is  denoted by
$\Span({\{x_i\}_{i=1}^n})$.
We write $x\circ y$ for the entrywise product of two
vectors $x,y$.

We denote the set of $d\times d$ Hermitian (resp. symmetric)
matrices by $\calh^d$ (resp.~$\calS^d$). An operator $X$ is called
an (orthogonal) {\em projector} if it satisfies $X=X^*=X^2$,
{where $X^*$ denotes  the conjugate transpose of $X$}. The
entrywise product of two matrices $X,Y$ is  denoted by $X\circ Y$
and their Kronecker product by $X\otimes Y$. Throughout this work we
equip $\calh^d$ with the Hilbert-Schmidt inner product $\la
X,Y\ra:=\tr(XY^*)$. The {\em direct sum} of two matrices  $X, Y$ is
the matrix $\left(\begin{smallmatrix} X& 0\\0 &
Y\end{smallmatrix}\right)$ which we   denote by $X \oplus Y$.  The
matrix with all entries equal to $1$ is denoted by $J$ and the
identity matrix  by $I$.

A matrix $X\in\calh^d$ is called {\em positive semidefinite} (psd) if
$\psi^*X\psi\ge 0$ for all $\psi\in \C^d$. The set of $d\times d$
Hermitian psd (resp. symmetric psd) matrices forms a closed convex
cone  denoted by $\calh^d_+$ (resp.~$\calS^d_+)$.

Let $\left(\mathcal{V},\la \cdot,\cdot\ra\right)$ be an inner
product space. The  {\em Gram matrix} of a family of vectors
$\{x_i\}_{i=1}^n\subseteq \mathcal{V}$, denoted by
$\gram({\{x_i\}_{i=1}^n}),$ is the  $n\times n$ matrix whose
{$(i,j)$} entry is given by $\la x_i,x_j\ra$, for all $i,j\in [n].$
Lastly, note that for all   $\{x_i\}_{i=1}^n\subseteq \mathcal{V}$
we have that $\gram( \{x_i\}_{i=1}^n)$ is psd and moreover, {$\rank
\left(\gram( \{x_i\}_{i=1}^n)\right)=\dim (
\Span(\{x_i\}_{i=1}^n))$}, {where $\dim (\mathcal{V})$ denotes  the dimension of
vector space $\mathcal{V}$}.

\subsection*{Quantum mechanics}

In this {section} we briefly introduce {some notions  from  quantum mechanics} that are of relevance  to this work. For a detailed introduction the interested  reader is  referred to \cite{NC00}.

According to the axioms of quantum mechanics, the \emph{state}  of a
{$d$-dimensional} quantum system is specified by a Hermitian psd
operator $\rho \in \calh^d_+$ (for some $d\ge~1$) such that
$\tr(\rho)=1$, {where $\tr(\rho)$ is the trace of $\rho$}. {In
order} to extract information from a quantum system   we need to
{\emph{measure}} it. Measurements on a quantum system are described
by the {Positive Operator Valued Measure} (POVM) formalism. A POVM
is  a family of psd  matrices ${\{M_i\}_{i=1}^m} \subseteq
\calh^d_+$ that sum to the identity operator, i.e., $\sum_{i=1}^m
M_i=I_d $. If the measurement ${\{M_i\}_{i=1}^m}$ is performed on a
quantum system {which is} in state $\rho$ then the outcome $i$ is
observed with probability $p_{i}:=\tr(\rho M_i)$. Note  that  from
the definitions above $\{p_i\}_{i=1}^m$  is  a valid probability
distribution.

 We also
 {use}
  a second (equivalent) mathematical  formalism  describing  a quantum  measurement. Given a $d$-dimensional quantum system, an  {\em observable} is any  Hermitian operator  $H$ acting on $\C^d$.  By the spectral theorem we know that $H=\sum_{i=1}^k\lambda_iP_i,$ where  $\{ \lambda_i\}_{i=1}^k$ ($k\le d)$ are the eigenvalues of $H$ and $\{P_i\}_{i=1}^k$ {are} the projectors {onto} the corresponding eigenspaces. The observable $H$ describes {the} POVM measurement {$\{P_i\}_{i=1}^k$} with outcomes $\{\lambda_i\}_{i=1}^k$, {i.e., upon} measuring state $\rho,$ the probability of  outcome $\lambda_i$ is given by  $\tr(\rho P_i)$.
  We say that $H$ is a {\em $\pm 1$  observable} if {it has $\pm 1$ eigenvalues.}

Consider two quantum systems ${\rm S_1}$ and ${\rm S_2}$ {and} say that ${\rm S_1}$ is in state $\rho_1 \in \calh^{d_1}_+$ and {${\rm S_2}$ is in state}
$\rho_2 \in \calh^{d_2}_+$.  In this case, the  state
of the joint system is given by the density matrix  $\rho_1\otimes \rho_2 \in {\calh^{d_1 d_2}_+}$. {If}
$\{M_i\}_{i=1}^{m_1}\subseteq \calh^{d_1}_+ $  and
$\{N_j\}_{j=1}^{m_2}\subseteq \calh^{d_2}_+ $ are POVMs  on the
individual systems ${\rm S_1}$ and ${\rm S_2},$ the {operators}
$\{M_i\otimes N_j :i\in [m_1], j\in [m_2]\}\subseteq
\calh_+^{d_1 d_2}$ define  a {\em {joint} measurement} on the
joint~system. {Note that not all states and measurements are of this form. {In particular, states that are not convex combinations of states of the form $\rho_1 \otimes \rho_2$ are said to be \emph{entangled}}.}

{We frequently consider \emph{rank $1$} quantum states which can be written as the outer product $\psi \psi^*$ for some vector $\psi$ (which must have unit norm since its outer product must have unit trace). Such quantum states are called \emph{pure} and there is one such pure quantum state we use frequently in this paper.} {We} denote by  $\Psi_d$ the  canonical {\em maximally entangled state} given by
\be\label{eq:maxentangled}
{1\over \sqrt{d}}\sum_{i=1}^de_i\otimes e_i\in \C^d\otimes \C^d.
\ee
{One can check that it is indeed entangled.}
We make repeated use of the fact that
\be\label{eq:maxent}
\Psi_d^*(A\otimes B)\Psi_d= \frac{1}{d} \, \tr\left(AB^\sfT\right), \text{ for all } A,B\in \C^{d\times d}.
\ee
The {\em Pauli matrices} are given by
\[
{I:=
\begin{pmatrix}
1 & 0 \\
0 & 1
\end{pmatrix}}, \;
X:=
\begin{pmatrix}
0 & 1 \\
1 & 0
\end{pmatrix}, \;
Y:=\begin{pmatrix}
0 & -i \\
i  & 0
\end{pmatrix},
\; \text{ and } \;
Z:=
\begin{pmatrix}
1 & 0 \\
0 & -1
\end{pmatrix}. \]
Note that the {(non-identity)} Pauli matrices are Hermitian, their trace is equal to zero,  they have $\pm 1$ eigenvalues and they pairwise anticommute. {Many of the explicit observables we consider in this paper are constructed using the Pauli matrices.}

\subsection*{Convexity}

A set $C \subseteq \R^{n}$ is {\em convex} if for all $a, b \in C$  and $\lambda \in [0,1]$ we have that  $\lambda a + (1-\lambda) b \in C$. A subset $F\subseteq C $ is called a {\em face} of $C$ if $\lambda c_1+(1-\lambda)c_2\in F$ implies that $c_1,c_2\in F$, for all $c_1,c_2\in C$ and $\lambda\in [0,1]$. We say that $c$ is an {\em extreme point} of the convex set $C$ if the set  $\{c\}$ is a face of $C$. We denote by ${\rm ext} (C)$ the set of extreme points of the convex set $C$.

\subsection*{Graph theory}

A graph $G$ is an ordered pair of sets $([n],E(G))$, where $E(G)$ is a collection of 2-element subsets of $[n]$. The elements of $[n]$ are called the {\em vertices} of the graph and the elements of $E(G)$  its {\em edges}. For every edge  $e=\{u,v\} \in E(G)$ we say that $u$ and $v$ are  {\em adjacent} and write $u\sim v$. A {\em subgraph} of $G$ is a graph whose vertex and edge sets are subsets of the vertex and {edge} sets of $G$, respectively. The {\em adjacency matrix} of $G$ is the $n\times n$  matrix
\[ {A:=\sum_{u\sim v}(e_ue_v^\sfT+e_ve_u^\sfT)}. \]
Note that the smallest eigenvalue of $A$ is negative.
The {\em support graph} of a  matrix $X\in\calS_n$, denoted by $S(X)$, is the  graph with vertex set $[n]$, {and} $u\sim v$ if and only if $X_{uv}\ne 0$ (and  $u \neq v$).  The {\em $n$-cycle}, denoted $C_n$,  is the graph with vertex set $[n]$ where $u\sim v$ if $(u-v)\equiv 1\mod n$.

\section{Studying the cpsd-rank}\label{sec:properties}

\subsection{Basic properties}\label{sec:basic}

{Our goal in this  section is to  determine basic    properties of the $\CPSDR$ that we use throughout this work.}

Note that  in the definition of $\CPSDR$ we only consider $\CPSD$-factorizations using  Hermitian (i.e., complex valued) psd matrices. If we  restrict to $\CPSD$-factorizations using  symmetric (i.e., real valued) psd matrices, we arrive at the notion of {\em real cpsd-rank}.  Nevertheless, the real cpsd-rank can differ at most by a  factor of two from the cpsd-rank.
To see this,  for  any $X\in \C^{d\times
d}$ set
 \be\label{eq:comtorealpsd}
T(X) := \dfrac{1}{\sqrt 2} \begin{pmatrix} \mathcal{R}(X) & -\mathcal{I}(X)\\ \mathcal{I}(X)&  \mathcal{R}(X)\end{pmatrix},
\ee
and note  that $T$ is a bijection   between $\C^{d\times d}$   and   $\R^{ 2d\times 2d}$. Furthermore,    $X\in \calh^n_+$  if and only if $T(X)\in \mathcal{S}^{2n}_+$  and moreover  $\inner{X}{Y} = \inner{T(X)}{T(Y)}$ for all $X,Y\in \calh^n_+$.

In our first result in this section we collect several simple properties concerning  the psd matrices  in a $\CPSD$-factorization.

\begin{lemma}\label{lem:fullrank}
Let $\{P_i\}_{i=1}^n\subseteq \calh^d_+$   be a $\CPSD$-factorization for $X\in \CPSD^n$.
\bi
\item[$(i)$] For any  $d\times d$ unitary  matrix  $U$, the matrices   $\{U^*P_iU\}_{i=1}^n\subseteq \calh^d_+$ are  a $\CPSD$-factorization of~$X$.
\item[$(ii)$] We have that $\CPSDR(X)\le \rank(\sum_{i=1}^nP_i)$.  In particular, if $\{P_i\}_{i=1}^n \subseteq \calh^d_+$ is a size-optimal $\CPSD$-factorization then  $\rank(\sum_{i=1}^nP_i)=d$.
\ei
\end{lemma}
\medskip 

\begin{proof} Part $(i)$ is clear. For $(ii)$
define the psd matrix $P:=\sum_{i=1}^n P_i$ and set  $r:=\rank(P)$.  Clearly,  $P$ is unitarily equivalent  to a diagonal matrix with  exactly  $r$  positive entries.
By restricting to the support of $P$, we get a   $\CPSD$-factorization of $X$ using $r\times r$ psd matrices.
\end{proof}

{Recall that any family of pairwise-commuting Hermitian   matrices  is  simultaneously diagonalizable by a unitary matrix (e.g. see \cite[Theorem 2.5.5]{HJ}).  Consider   $X\in \CPSD^n$  and let $\calI\subseteq [n]$ so that the principal submatrix corresponding to  $\calI$ is  diagonal (with positive diagonal entries). 
In view of Lemma~\ref{lem:fullrank}  $(i)$, 
 we may assume that in any $\CPSD$-factorization $\{P_i\}_{i=1}^n$ of $X$, the matrices $\{P_i\}_{i\in \calI}$  can be taken to be diagonal~psd.}
 {This immediately implies $\CPSDR(I_n) \geq n$, for any $n$, which can be easily seen to hold with equality.}
We proceed with a second example.} 
 
\medskip\begin{example}
{We prove that the cpsd-rank of the
matrix $\begin{pmatrix}2 & 0 & 0 & 1 & 1\\
0 & 2 & 0 & 1 & 1\\
0 & 0 & 2 & 1 & 1\\
1 & 1 & 1 & 3 & 0\\
1 & 1 & 1 & 0 & 3\end{pmatrix}$} 
is equal to $4$ and {a} size-optimal $\CPSD$-factorization is given by 
$$
\left(\begin{smallmatrix}\sqrt{2} & 0 & 0 & 0\\
0 & 0 & 0 & 0\\
0 & 0 & 0 & 0\\
0 & 0 & 0 & 0\end{smallmatrix}\right),\left(\begin{smallmatrix}0 & 0 & 0 & 0\\
0 & 1 & 0 & 0\\
0 & 0 & 1 & 0\\
0 & 0 & 0 & 0\end{smallmatrix}\right),\left(\begin{smallmatrix}0 & 0 & 0 & 0\\
0 & 0 & 0 & 0\\
0 & 0 & 0 & 0\\
0 & 0 & 0 & \sqrt{2}\end{smallmatrix}\right),\left(\begin{smallmatrix}{1\over \sqrt{2}} & 0 & 0 & {1\over \sqrt{2}}\\
0 & 1 & 0 & 0\\
0 & 0 & 0 & 0\\
{1\over \sqrt{2}} & 0 & 0 & {1\over \sqrt{2}}\end{smallmatrix}\right), \left(\begin{smallmatrix}{1\over \sqrt{2}} & 0 & 0 & -{1\over \sqrt{2}}\\
0 & 0 & 0 & 0\\
0 & 0 & 1 & 0\\
-{1\over \sqrt{2}} & 0 & 0 & {1\over \sqrt{2}}\end{smallmatrix}\right).
$$
%Next,  
{Assume that $\CPSDR(X)\leq3$ and let  $\{P_i\}_{i=1}^5\subseteq \calh^3_+$ be a $\CPSD$-factorization. Since $P_1, P_2$ and $ P_3$  commute pairwise, by  applying an appropriate change of basis,  we may assume by  Lemma \ref{lem:fullrank} $(i)$ that  they are diagonal. Furthermore, as $\la P_i,P_j\ra=0,$ for $i\ne j\in [3]$ it follows that $P_i=\sqrt{2}e_ie_i^\sfT$, for  $1\le i\le 3$. Since $P_4$, $P_5$ are $3\times 3$ orthogonal psd matrices (and nonzero) one of them has rank $1$. Suppose without loss of generality $P_4 = xx^*$ for some $x=(x_i) \in \C^3$. Note  that $|x_i|^2 = 1/\sqrt{2},$ for all $1\le i\le 3$. Thus, $\langle P_4, P_4 \rangle = (\sum_{i=1}^3 |x_i|^2)^2 = 9/2 \neq 3$, a contradiction.} 
\end{example}
\medskip

In the remaining part of this section we focus on matrix operations  that preserve the property of being cpsd and we investigate in what way they affect the $\CPSDR$.
\vspace{0.1cm}

\begin{lemma}\label{lem:whatever}
Consider  $X \in \CPSD^n$. We have that:
 \bi
 \item[$(i)$] For any $n\times n$ diagonal matrix $D$ with strictly positive diagonal entries     we~have
 $$DXD\in \CPSD^n, \text{ and } \CPSDR(X)=\CPSDR(DXD).$$
 \item[$(ii)$] For any $n\times n$  permutation matrix $P$ we have
 $$PXP^\sfT\in \CPSD^n, \text{ and } \CPSDR(X)=\CPSDR(PXP^\sfT).$$
 \ei
\end{lemma}
\medskip

We now determine  how the $\CPSDR$ behaves under matrix sums.

\medskip
\begin{lemma}\label{cdevgerger}
For any $X,Y\in \CPSD^n$ we have that $X+Y\in \CPSD^n$ and furthermore,
\[ \CPSDR(X+Y) \leq \CPSDR(X) + \CPSDR(Y). \]
\end{lemma}

\begin{proof}Let $\{P_i\}_{i=1}^n\subseteq \calh^{d_1}_+$ and $\{Q_j\}_{j=1}^m\subseteq \calh^{d_2}_+$ be size-optimal $\CPSD$-factorizations for $X$ and $Y$, respectively.  For all $i\in [n]$ define $Z_i:=P_i\oplus Q_i\in \calh^{d_1+d_2}_+$ and note that the matrices  $\{Z_i\}_{i=1}^n$ are a $\CPSD$-factorization for $X+Y$.
\end{proof}
\medskip 

\begin{remark}As it turns out, the cpsd-rank of the sum of a family {of} cpsd matrices can be exponentially smaller compared to {any} of the individual cpsd-ranks. To see this, let
{$X\in \CPSD^n$} and define  $X_{sym}:=\sum_{P\in P_n}PXP^\sfT,$  where  $P_n$ is the set of $n\times n$ permutation matrices. By Lemma~\ref{lem:whatever} we have   $X_{sym}\in \CPSD^n$ and  by its definition {we have} $X_{sym}=(a-b)I+bJ,$  for appropriate constants  $a,b$ where $a\ge b\ge 0$. By Lemma~\ref{cdevgerger} we have  $\CPSDR(X_{sym})\le n+1$, {since $\CPSDR(I_n) = n$ and $\CPSDR(J) = 1$}. On the other hand, in Section \ref{sec:explowerbounds} we show that  for any $n\ge1,$ there exists a matrix in $\CPSD^{2n}$ with  cpsd-rank  $2^{\Omega(\sqrt{n})}$.
\end{remark}
\medskip 

In our next result  we determine how the  $\CPSDR$ behaves under direct sums.

\medskip 

\begin{lemma}For any $X\in \CPSD^n$ and $Y\in  \CPSD^m$ we have that $X\oplus Y\in \CPSD^{n+m}$ and furthermore,  $\CPSDR(X\oplus Y)=\CPSDR(X)+\CPSDR(Y)$.
\end{lemma}
\medskip

\begin{proof}Let $\{P_i\}_{i=1}^n\subseteq \calh^{d_1}_+$ and $\{Q_j\}_{j=1}^m\subseteq \calh^{d_2}_+$ be size-optimal $\CPSD$-factorizations for $X$ and $Y$, respectively. For $i\in [n]$, set $\tilde{P}_i:= {P_i\oplus 0_{d_2}} \in \calh^{d_1+d_2}_+$  and for $j\in [m]$ set $\tilde{Q}_j:= {0_{d_1}\oplus Q_j} \in \calh^{d_1+d_2}_+$. Clearly the matrices $\{\tilde{P}_i\}_{i=1}^n\cup \{\tilde{Q}
_j\}_{j=1}^m$ form a $\CPSD$-factorization for $X\oplus Y$. Thus we get  that  $X\oplus Y\in \CPSD^{n+m}$ and  furthermore, $\CPSDR(X\oplus Y)\le \CPSDR(X)+\CPSDR(Y)$.

For the reverse  inequality  set $d:=\CPSDR(X\oplus Y)$ and let   $\{P_i\}_{i=1}^{n}\cup\{Q_j\}_{j=1}^m\subseteq \calh^d_+$ be a size-optimal   $\CPSD$-factorization for $X\oplus Y$. Moreover, set $P:=
\sum_{i=1}^nP_i$ and $Q:=\sum_{j=1}^mQ_j$.  By Lemma~\ref{lem:fullrank} we have     $ \rank  (P+Q)=d.$
Furthermore, by the structure of $X\oplus Y$  we have  $\la P_i,Q_j\ra =0,$ for all $i,j$ and thus $\la P,Q\ra=0$. As  $P,Q$ are psd this implies that  $d=\rank(P+Q)=\rank(P)+\rank(Q).$
Since the matrices $\{P_i\}_{i=1}^n$ form a $\CPSD$-factorization of $X$, by Lemma \ref{lem:fullrank} $(ii)$  we have  $\rank(P) \ge \CPSDR(X)$  and similarly
 that $\rank(Q)\ge \CPSDR(Y).$ Putting everything together, the claim~follows.
\end{proof}
\medskip

Our next goal is to show that there exist  $\CPSD$ matrices  that do
not admit $\CPSD$-factorizations using only rank-one factors. In contrast to this,
  restricting to  factorizations using  rank-one  psd matrices has been a  useful approach to provide upper bounds on the
{positive semidefinite rank (cf. Section \ref{psdvscpsd})}
\cite{WL,FGPRT}.

We denote by $\CPSD^{n,1}$ {the set of  matrices} in
$\CPSD^n$ that admit $\CPSD$-factorizations using rank-one factors.
Furthermore, we call a {\em Hadamard square root} of $X\in
\R^{n\times m}_{+}$ any matrix obtained by replacing each entry of
$X$ by one of its two square roots. We have the following result
whose proof is straightforward and is omitted.
\medskip

\begin{lemma}\label{lem:squareroot}
For any  matrix  $X\in \calh^n_+$
we have that
$X\circ X^*\in
\CPSD^{n,1}$   and   moreover  $\CPSDR({X\circ X})\le \rank(X)$.  In particular, if $X\in \calh^n_+$ is a matrix with  0/1 entries   then~$X\in \CPSD^n$ and $\CPSDR(X)\le \rank(X)$. Conversely, if
$X\in \CPSD^{n,1}$ then $X$ has a psd Hadamard square root.
\end{lemma}
\vspace{0.1cm}

As a concrete example of a matrix {in} $\CPSD\setminus \CPSD^1$, {consider}  \be X=\begin{pmatrix}
1 & \sqrt{2}/2 & \sqrt{2}/2\\
\sqrt{2}/2 & 1 & 1/10\\
\sqrt{2}/2 & 1/10 & 1
\end{pmatrix}.
\ee
Clearly  $X\in \CPSD^3 {= \DNN^3}$, but  no Hadamard square root of $X$ is psd.

\subsection{Lower bounds}\label{sec:lowerbounds}

In this section we derive two general lower bounds on the $\CPSDR$.
The first one is analytic  and the second one is based on the
support of the matrix. We show that in both cases, our  bounds never
exceed the size of the~matrix.

\subsubsection{Analytic lower bound} We start with  the following result. 

\medskip

\begin{theorem}\label{thm:analyticlowerbound}
For any matrix $X\in \CS^n$ we have that
\be\label{eq:analyticlowerbound}
\CPSDR(X)\ge { \left({\sum_{i =1}^n} \sqrt{X_{ii}}\right)^2\over {\sum_{i,j = 1}^n} X_{ij}}.
\ee
\end{theorem}

\begin{proof}{Set  $d:= \CPSDR(X)$ and let   $\{P_i\}_{i=1}^n \subseteq \calh^d_+$ be a size-optimal  $\CS$-factorization}. By Lemma \ref{lem:fullrank}  $(ii)$  we have that  $P:=\sum_{i=1}^nP_i\in \calh^d_+$ has full-rank. By the  Cauchy-Schwartz inequality  we have  that  $d\ge \tr(P)^2/\tr(P^2)$.
Note  that $\tr(P^2)
=\sum_{i,j=1}^nX_{ij}$.~Lastly,
\[
\tr(P)^2=\left(\sum_{i=1}^n\tr(P_i)\right)^2\ge\left(\sum_{i=1}^n\sqrt{\tr(P_i^2)}\right)^2 {= \left( \sum_{i=1}^n \sqrt{X_{ii}} \right)^2},
\]
where we used $\tr(P_i)\ge\sqrt{\tr(P_i^2)},$ since $P_i\in\calh^d_+$, {for the last inequality}.
\end{proof}
\vspace{0.1cm}

In view of Theorem \ref{thm:analyticlowerbound}, two  remarks are in
order. First, it follows by \eqref{eq:analyticlowerbound} that
$\CPSDR(I_n)\ge n$ and {this is} obviously  tight.
{Second}, the Cauchy-Schwartz inequality combined with the fact
{that} any cpsd matrix is entrywise nonnegative {implies} that
the lower bound \eqref{eq:analyticlowerbound}  can never exceed  the
size of the matrix.

\subsubsection{Support-based lower bound}\label{sec:supprotlowerbound}

To study  support-based lower bounds on the cpsd-rank we introduce
the following graph parameter: \be\label{eq:supportlowerbound}
f(G):=\min \{ d\ge 1:\exists   \text{ subspaces }
\{\call_i\}_{i=1}^n\subseteq \C^d \text{ s.t. } \call_i\perp \call_j
\Longleftrightarrow i\not \sim j \}. \ee

To see $f(G)$  is well-defined let $A$ be the adjacency matrix of $G$ and  let $\tau$ be its least eigenvalue with multiplicity $m$. Since  $A-\tau I\in \calS^n_+$ and  $\rank(A-\tau I )= n-m,$  there exist   vectors $\{x_i\}_{i=1}^n\subseteq \R^{n-m}$     such that $A-\tau I=\gram(\{x_i\}_{i=1}^n)$. For $i\in [n]$, set   $\call_i:=\Span({\{ x_i \}})$ and note this is  a feasible solution for~\eqref{eq:supportlowerbound} yielding $f(G) \leq n-m$.

\medskip 

\begin{theorem}\label{prop:supportlowerbound}
For any graph  $G=([n],E)$ we have that $f(G)$ is equal to
\be\label{eq:supportbasedbound}
 \min \{\CPSDR(X): X\in\CPSD^n \text{ and }S(X)=G   \}.
\ee
\end{theorem}

\begin{proof}
By Lemma~\ref{lem:squareroot}, the   0/1  matrix $A-\tau I$  defined in the previous paragraph is cpsd. This shows that \eqref{eq:supportbasedbound} is feasible.  Let  $X$ be optimal  for  \eqref{eq:supportbasedbound}  and let {$\{P_i\}_{i=1}^n\subseteq \calh^d_+$} be a size-optimal  $\CPSD$-factorization for $X$. For $i\in [n]$, define  {$\call_i:={\rm Range}(P_i)\subseteq \C^d$} and note this is  feasible for \eqref{eq:supportlowerbound}.
Conversely, let $\{\calL_i\}_{i=1}^n\subseteq \C^d$ be a family of subspaces feasible for \eqref{eq:supportlowerbound} and for $i\in [n]$ define $P_i$ to be the orthogonal projector {onto} $\calL_i$. Lastly, note that the matrix   $X:=\gram(\{P_i\}_{i=1}^n)\in \CPSD^n$ is feasible for \eqref{eq:supportbasedbound} and satisfies  $\CPSDR(X) \leq d$.
\end{proof}
\vspace{0.1cm}

By {Theorem}  \ref{prop:supportlowerbound} and the fact that $f(G)$  is upper bounded  by $n$  it  follows   that {support-based} lower bounds on the $\CPSDR$ never exceed the size of the~matrix.

\subsection{{{Comparisons} with other notions of rank}}\label{sec:comparison}

In this section we investigate further the {relationships} between the cpsd-rank and other
notions of matrix ranks.

\subsubsection{cpsd-rank vs. rank} \label{ssr}

As $\calh^d$ is isometrically isomorphic to $\R^{d^2}$,  we~have
\be
\label{eq:rank}
\sqrt{\rank(X)}\le \CPSDR(X),
\ee
for any $X\in \CPSD$.
We provide an example that illustrates that the above can be tight up to a constant factor. Let $r \geq 2$ be an integer and let
$E_{i,j}:=I_r+e_ie_j^\sfT+e_je_i^\sfT\in \calh^r_+$ for all $i,j\in [r]$. The matrix $X:=\gram(\{E_{i,j}\}_{i,j})\in \CPSD^{r(r-1)/2}$
has $\CPSDR(X)\le~r$, by construction, while $X$ can be easily seen to have full rank.
On the other hand, no upper bound for $\CPSDR(X)$ in terms of $\rank(X)$ is known.

\subsubsection{cpsd-rank vs. psd-rank} \label{psdvscpsd}

Given any entrywise nonnegative matrix $X \in \R^{n \times m}_{+}$,
its {\em positive semidefinite rank} {(psd-rank)}, denoted by
$\PSDR(X)$, is defined as the least integer $d\ge 1$ for which there
exist $\{A_i\}_{i=1}^n, \{B_j\}_{j=1}^m\subseteq \calh^d_+$ such
that $X_{ij}=\tr(A_iB_j)$ for all $i\in [n],j\in [m]$. Generalizing
a theorem by  Yannakakis~\cite{Yan}, it was shown in \cite{FMPTW}
and \cite{GPT} that $\PSDR(S)$ where $S$ is a slack matrix for
polytope $P$ corresponds to the smallest size of a spectrahedron
that projects onto $P$.  For further properties of the $\PSDR$ the
reader is referred to   \cite{FGPRT} and Section \ref{psdvscpsd}.

Clearly, for  any $X\in \CPSD^n$ we have that $\PSDR(X)\le \CPSDR(X)$. Furthermore, since  $\PSDR(X) \leq n,$ for any $X\in \CPSD^n$,  the example of the matrix $X$ with $\CPSDR(X) = 2^{\Omega(\sqrt{n})}$ given  in Section 5  provides an exponential separation between $\PSDR(X)$ and $\CPSDR(X)$.

We conclude this section by determining  the exact relation between
$\PSDR(X)$ and $\CPSDR(X)$. This  follows from the connection to
Bell scenarios. As both quantities are invariant under scaling by a
positive constant, {without loss of generality} we can assume
that $\sum_{i,j}X_{i,j}=1$ so that $\pbf:=(X_{ij})_{ij}$ is a
probability distribution. We can think of $\pbf$  as {a} behavior
corresponding to a $(1,1, m, n)$ Bell scenario where each party has
a unique POVM. As mentioned in the introduction, in this case the
behavior $\pbf$  is
 quantum and moreover,  $\mathcal{D}(\pbf)=\PSDR(X)$ \cite{JSWZ}. This fact combined   with  Theorem~\ref{thm:minsizequantumcorrelation}
 {implies} that $\PSDR(X)$ is equal to $$\min \left\{ \CPSDR(R):
R=\begin{pmatrix}A & X\\X^\sfT & B\end{pmatrix}   \in \CPSD^{n+m} \;
\text{ and } \;
\sum_{i,j=1}^nA_{ij}=\sum_{i,j=1}^nB_{ij}=1\right\}.$$

 In turn, this is equal to the smallest integer $d\ge 1$ for which there exist Hermitian psd matrices  $\{A_i\}_{i=1}^n, \{B_j\}_{j=1}^m\subseteq \calh^d_+$ such that $X_{ij}=\tr(A_iB_j),$ for all $i\in [n],j \in [m]$ and $\sum_{i=1}^n A_i=\sum_{j=1}^m B_j$. As  a corollary we also get that in any psd-factorization of $X$ we may assume that the psd  factors  satisfy $\sum_{i=1}^n A_i=\sum_{j=1}^m B_j$.

\subsubsection{cpsd-rank  vs.  cp-rank}

For all matrices $X\in \CP$ we clearly have   that
$$\Omega(\CPR(X)^{1/4}) \leq \CPSDR(X)\le \CPR(X).$$

The lower bound follows from  the fact that $\CPR(X)\le \binom{\rank(X)+1}{2}-1,$ for all $X\in \CP$ (e.g. see   \cite[Theorem 3.5]{CP}) combined with \eqref{eq:rank}.

We now give an example where $\CPSDR(X)=\CPR(X)$.  For this, let  $a\in~(0,3/4)$ and set $X_a:=I_3+ae_1e_3^\sfT+ae_3e_1^\sfT$.
 Recall that  $\CPSD^3=\CP^3=\DNN^3$. By \cite[Theorem 3.2]{CP} we have that $\CPR(X_a)=\rank(X_a)=3$.  From Theorem \ref{thm:analyticlowerbound} it follows that $\CPSDR(X_a)\ge 3$, thus $\CPSDR(X)=\CPR(X)$ for this case.

Lastly, the  example given in Section \ref{ssr} also provides a quadratic separation between the cp-rank and the cpsd-rank.
The matrix $X \in \CP$ as it is the Gram matrix of $E_{i,j}$ which has {nonnegative} entries. Further, $\CPR(X) \geq \rank(X)=\binom{r}{2}$ while $\CPSDR(X) \leq r$, by construction.

\section{Gram-Lorentz matrices}\label{sec:gramlorentz}

As already mentioned,  it is  currently not known  whether
the $\CPSDR$ of all  matrices in $\CPSD^n$ admits a finite upper bound. In this section we identify a family of $\CPSD$ matrices for which it is possible to prove a finite upper bound. These  are the {\em Gram-Lorentz matrices}, i.e.,  Gram matrices of Lorentz cone vectors. In Section~\ref{sec:embedding} we recall a construction from  \cite{FW}   where it is shown that the Lorentz cone can be
isometrically embedded into a psd  cone  of an
appropriate size.  This implies that Gram-Lorentz matrices are cpsd.  In Section \ref{sec:glmatrices}  we show  that the $\CPSDR$ of a Gram-Lorentz matrix is upper bounded in terms of its rank. Lastly, in Section~\ref{sec:cpsdcpseparation} we use Gram-Lorentz matrices  to construct matrices in $\CPSD\setminus \CP$, generalizing a construction from \cite{FW}.

\subsection{Embedding the Lorentz cone isometrically into $\calh^d_+$}\label{sec:embedding}

Underlying the results in this section is  a linear embedding of  vectors in $\R^n$  into  traceless Hermitian operators of size $ 2^{\left\lfloor \frac{n}{2}\right \rfloor},$ so that inner products are preserved up to a constant factor and unit vectors get mapped to $\pm 1$ observables.

In fact, this embedding corresponds to a   complex  representation of the Clifford algebra over $(\R^n,\la \cdot,\cdot\ra)$ defined in terms of the so-called  {Brauer-Weyl} matrices. For more details see \cite{GW} and Section \ref{sec:Clifford} in the Appendix.
This embedding is also the main  ingredient in Tsirelson's characterization of binary outcome correlations~\cite{TS87}.
\vspace{0.1cm}

\begin{theorem}\label{thm:cliffordalgebramain}
There exists a linear map $\gamma: \R^n\rightarrow \calh^d,  $ where    $d=2^{\left\lfloor \frac{n}{2}\right \rfloor}$   such that:
\bi
\item[$(i)$] For all $x\in \R^n$ we have that $\tr\left(\gamma(x)\right)=0$;
\item[$(ii)$]   For all $x\in \R^n$ with {$\|x\|=1$}  we have $\gamma(x)^2=I_d$;
\item[$(iii)$] For all $x,y\in \R^n$  we have  $d\cdot \la x,y\ra=\tr\left( \gamma(x)\gamma(y)\right)$.
\ei
\end{theorem}
\medskip

Specifically, when   $n=2\ell$ we define:
 \be\label{eq:firsthalf}
 \gamma(e_i)= Z^{\otimes(i-1)}\otimes X\otimes I_2^{\otimes (\ell-i)} \in \calh^d,  \ (i\in [\ell]),
 \ee
 and
 \be\label{eq:secondhalf}
 \gamma(e_{i+\ell})= Z^{\otimes(i-1)}\otimes Y\otimes I_2^{\otimes (\ell-i)}\in \calh^d, \ (i\in [\ell]).
\ee
When  $n=2\ell+1$ we define $\{\gamma(e_i)\}_{i=1}^{2\ell}$ as in \eqref{eq:firsthalf} and \eqref{eq:secondhalf}  and  set $\gamma(e_{2\ell+1})= Z^{\otimes \ell}.$  Lastly, we extend $\gamma$ linearly, i.e.,  $\gamma(x)=\sum_{i=1}^nx_i\gamma(e_i),$ for any  $x=\sum_{i=1}^n x_i e_i\in \R^n$.

By definition of $\gamma$ it follows that $\tr(\gamma(x))=0$ for all
$x\in \R^n$. Furthermore, note  that
$\gamma(e_i)\gamma(e_j)+\gamma(e_j)\gamma(e_i)=2\delta_{ij}I_d, $
for all $i,j$, which implies that \be\label{eq:clifford}
\gamma(x)\gamma(y)+\gamma(y)\gamma(x)=2\la x,y\ra I_d, \text{ for
all } x,y\in \R^n. \ee For any $x\in \R^n$ with ${\|x\|=1}$, setting
$x=y$ in \eqref{eq:clifford} we get that $\gamma(x)^2=I_d$. Lastly,
taking traces in \eqref{eq:clifford} we see that ${d\cdot\la
x,y\ra}=\tr\left( \gamma(x)\gamma(y)\right)$ for all $x,y\in \R^n$.

Next we introduce a convex cone which plays a central role in this work.
\medskip
\begin{definition}The {$m$-dimensional} {\em Lorentz cone}, denoted $\calL_m$, is defined as the set of vectors in $ \R^m$ whose angle with the vector $e_1\in \R^m$ does not exceed $\pi/4$, i.e.,
\[ \calL_m=\big \{(c,x)\in \R\times  \R^{m-1}: c\ge \|x\|\}.\]
\end{definition}

It was shown  in \cite{FW}  that the Lorentz cone can be
isometrically embedded into the cone of psd  matrices of an
appropriate dimension. For the convenience of the reader, we
include a short new proof of the existence of the isometric embedding.

\medskip 

\begin{theorem}[\cite{FW}]\label{thm:isometry} Set $d:= {2^{\lfloor \frac{m-1}{2} \rfloor}}$.
There exists {an} isometry $\Gamma: \real^m \rightarrow \calh^d,$  such that
 $$\calL_m=\{(c,x)\in {\R \times \R^{m-1}}: {\Gamma((c,x))} \in \calh^d_+\}.$$
\end{theorem}

\begin{proof}
For $(c,x)\in \R^m=\R\times  \R^{m-1}$ define
\be\label{eq:isometry}
{\Gamma((c,x))}={1\over \sqrt{d}}\left(cI_d+\gamma(x)\right).
\ee
To see that $\Gamma$ defines  an isometry note that
\[
{
\tr ( \gamma((c,x)) \gamma((c',x')))
=
cc' + \inner{x}{x'}
=
\inner{(c,x)}{(c',x')}.
}
\]

Lastly, recall that $\gamma(x)^2=\|x\|^2I_d$  and  thus the eigenvalues of $\gamma(x)$ are given by ${\pm \|x\|}$.  Consequently,
$\Gamma((c,x))\in \calh^d_+$  if and only if $  c \ge {\|x\|}$.
\end{proof}
\vspace{0.1cm}

For concreteness,  below  we explicitly describe  the isometry $\Gamma: \R^3\rightarrow~\calh^2$.
\medskip 

\begin{example} \label{ex:23isometry}
Let $(c,v,w)\in \calL_3$ (so $v,w\in \R)$.  By Theorem  \ref{thm:cliffordalgebramain} we have that
$$\gamma(v,w)=vX+wY=\begin{pmatrix}
0 & v-iw\\v+iw & 0
\end{pmatrix}.
$$

Thus, substituting this  into \eqref{eq:isometry}  we see that
\be\label{eq:32phi}
\Gamma(c, v, w)={1\over \sqrt{2}}\begin{pmatrix}
c & v-iw\\v+iw & c
\end{pmatrix}.
\ee Note that ${\Gamma((c, v, w))} \in \calh^2_+$   since $c \geq 0$ and the
determinant $c^2-(|v|^2+|w|^2) \geq 0$ since $(c,v,w)\in \calL_3$. Lastly, notice that if $(c,v,w)$ lies  on the boundary of $\calL_3$, then the determinant is $0$ and thus ${\Gamma((c, v, w))}$ has rank 1.
\end{example}

\subsection{Gram-Lorentz matrices}\label{sec:glmatrices}

{Theorem \ref{thm:isometry}  suggests  the following definition.}
\medskip 
\begin{definition} A
matrix $X {\in \Herm^n}$ is called  {\em Gram-Lorentz} if  there exist  vectors $\{\ell_i\}_{i=1}^n\subseteq \calL_m$ (for some $m\ge 1$) such that $X=\gram(\{\ell_i\}_{i=1}^n).$
We denote  the set of $n\times n$ Gram-Lorentz matrices by~$\CL^n$.
\end{definition}
\medskip 

The study of the set of Gram-Lorentz matrices is motivated as follows.  Firstly, in view of Theorem~\ref{thm:isometry}, we have  that
 $ \CL^n\subseteq \CPSD^n$. Identifying matrices in {$\CL\setminus~\CP$} therefore provides a systematic approach for finding matrices
{in} $\CPSD \setminus \CP$. All known examples of matrices in $\CPSD \setminus \CP$ are  constructed exactly in this manner \cite{FW}.
Secondly,  as we show  in this {section}, the $\CPSDR$ of a Gram-Lorentz matrix  can be upper bounded in terms of its rank. Since it {is} currently unknown whether the $\CPSDR$ of all  matrices in $\CPSD^n$  admits a finite upper bound, it is instructive
to identify families of matrices in $\CPSD^n$ for which there is {one}.

It is not  clear from its definition whether $\CL$ is convex. In fact, the following~holds.

\medskip
\begin{lemma} The set
$\CL^n$  is convex if and only if $n\le 2$.
\end{lemma}
\medskip 

\begin{proof}
First we show that $\CL^2=\DNN^2$. For this, let   $A := \left( \begin{smallmatrix} a & b \\ b & c \end{smallmatrix} \right)\in~\DNN^2$. We now show how to write it as the Gram matrix of vectors in $\calL_3$. Note that $ac\ge b^2$ and wlog  assume that  $a \geq c >0.$
Set
$ v_1:=  \sqrt{\frac{a}{2}}\left( 1,1, 0 \right)$ and $ v_2: =  \sqrt{\frac{c}{2}} \left(1, d , \sqrt{1-d^2} \right),$
where $d: = (2b - \sqrt{ac})/\sqrt{ac}$.
Lastly, note   that $A=\gram(\{ v_1, v_2 \})$ and that $\{v_i\}_{i=1}^2\subseteq \calL_3$.
Since $\CL^2 \subseteq \DNN^2$, the two sets are equal, and thus $\CL^2$ is~convex.

Next we show that  $\CL^n$ is not convex for $n\geq 3$. Since $\{ e_ie_i^\sfT \}_{i=1}^n\subseteq   \CL^n$  we have that $2 I_n$ is in the convex hull of $\CL^n$.
It {is sufficient} to show  that $2 I_n \not\in \CL^n$ for $n \geq 3$. To
this end, suppose that $2 I_n$ is the Gram {matrix} of the Lorentz vectors
$\{(t_i, u_i)\}_{i=1}^n$. This implies that $t_i^2 + \norm{u_i}_2^2 =
2,$ for all $i\in [n]$ and  $t_i t_j + \inner{u_i}{u_j} = 0$ for all
$i\ne j\in [n]$. Since $t_i\ge \|u_i\|$ for all $i\in [n]$, the
Cauchy-Schwarz inequality  implies  that    $t_i t_j \geq \norm{u_i}_2
\norm{u_j}_2 \geq |\inner{u_i}{u_j}| =t_i t_j,$  for all $i\ne j\in
[n]$. Thus,  equality holds throughout which shows  that  $t_i =
\norm{u_i}_2 = 1,$ for all $i\in [n]$, and that $u_i=-u_j$ for all
$i\ne j\in [n]$. This gives a contradiction since $n \geq 3$.
\end{proof}

We now show that for any  $\CL$ matrix we can place an upper  bound on the  dimension of the Lorentz cone we need to generate it.

\medskip 
\begin{lemma}\label{prop:lorentzbound}
Any   $X\in \CL^n$ has  a $\CL$-factorization using vectors in $\calL_{\rank(X)+2}.$
\end{lemma}
\medskip

\begin{proof}
Since $X\in \CL^n$ there exist vectors  $\{\ell_i\}_{i=1}^n\subseteq \calL_m$ (for some $m\ge 1$) such that $X= \gram( \{ \ell_i \}_{i=1}^n )$. For $i\in [n]$ set  $\ell_i:=(t_i,u_i)$, where $u_i\in \R^{m-1}$ and $\|u_i\| \le t_i$.  Define $U:= \gram( \{ u_i \}_{i=1}^n )$ and $t:=\sum_{i=1}^n t_ie_i$ and note that  $X=U+tt^\sfT$. Since $U$ is psd of rank at most $r:=\rank(X)+1$, there exists a family of vectors $\{\tilde{u}_i\}_{i=1}^n\subseteq \R^r$ such that   $U= \gram( \{ \tilde{u}_i \}_{i=1}^n )$. Lastly, since $\| \tilde{u}_i \| = \| u_i \|$ for all $i\in [n]$ it follows that  the vectors $\tilde{\ell}_i:=(t_i,\tilde{u}_i)$ lie in {$\calL_{r+1}$} and satisfy $ X= \gram( \{ \tilde{\ell}_i \}_{i=1}^n )$.
\end{proof}
\medskip 

\begin{theorem}\label{upperboundcpsd}
For any matrix $X\in \CL^n$ we have that  $X\in \CPSD^n$ and
\be\label{cor:glupperbound}
 \CPSDR(X)\le 2^{\lfloor (\rank(X)+1)/2\rfloor}.
\ee
\end{theorem}

\begin{proof}
The proof follows by combining Theorem \ref{thm:isometry} with Lemma \ref{prop:lorentzbound}.
\end{proof}
\vspace{0.1cm}

In Section \ref{sec:explowerbounds} we show that this bound is essentially tight (cf. Remark \ref{rem:tightlowerbound}).

\subsection{Matrices in $\CPSD\setminus \CP$}\label{sec:cpsdcpseparation}

In this section  we use  Gram-Lorentz matrices {to present} a new family of matrices in
$\CPSD\setminus \CP$.
{To this end, we make use of the following technical lemma.}

\medskip 
\begin{lemma}\label{lem:Gramdecomp}
Consider vectors $\{p_i\}_{i=1}^n\subseteq \R^d$ and scalars
$\{\lambda_i^j: i\in [n],  j\in[m]\}$  such that $c:=\sum_{i=1}^n
\lambda^j_i p_i$ for all $j\in[m]$. Consider vectors
$\{q_i\}_{i=1}^n\subseteq \R^{d'}$ satisfying  $\la p_i,p_j\ra=\la
q_i,q_j\ra$ for  all $i,j\in [n]$. Then there exists $c'\in \R^{d'}$
such that    ${\|c\|=\|c'\|}$ and moreover  $c'=\sum_{i=1}^n
\lambda^j_i q_i$ for all $j\in [m]$.
\end{lemma}
\medskip

\begin{proof}
For  $j\in [m]$ set  $c_j':= \sum_{i=1}^n\lambda^j_i q_i.$  For
$j\ne j' \in[m]$ we have  ${\| c'_j-c'_{j'}\|^2} =0$ and thus
$c'_j=c'_{j'}$. Lastly, set   $c'$ to be this common  value and note
that ${\|c'\|^2 = \|c\|^2}$.
\end{proof}
\vspace{0.1cm}

{We now give} a sufficient condition for constructing matrices in $\DNN\setminus \CP$, generalizing the construction in~\cite{FW}.

\medskip 
\begin{theorem}
\label{thm:necessarycondnotcp}
Consider vectors $\calF:=\cup_{i\in I}\{p_i,p_i'\}$ with the following properties:
\vspace{0.1cm}

\bi
\item[$(i)$] There exists a {nonzero} vector $c$ such that $(p_i+p_i')/2=c,$ for all $i\in I$;
\item[$(ii)$] For all $i\in I$ we have $\la p_i,p'_i\ra=0$;
\item[$(iii)$] There exists $J\subseteq I$ that has odd cardinality and $\sum_{j\in J} p_j=c\cdot |J|$;
\item[$(iv)$] The pairwise inner products of all vectors in $\calF$ are nonnegative.
\ei
\vspace{0.1cm}
Then  we have that $\gram(\mathcal{F})\in \DNN\setminus \CP$.
\end{theorem}

\vspace{0.1cm} 

\begin{proof}
By $(iv)$ we have  $\gram(\mathcal{F})\in \DNN$. {For}
 a contradiction, assume   that  $\gram(\mathcal{F})\in~\CP$  and let $ \{a_f\}_{f\in \calF}\subseteq\R^d_+$ be a nonnegative   Gram factorization.

By  Lemma~\ref{lem:Gramdecomp} there exists a  vector $a\in \R^d$ with ${\|a\| = \|c\|}$ satisfying   $(a_{
i}+a'_{i})/2=~
a,$ for all $i\in I$ and $\sum_{j\in J}a_j=\lvert J\rvert a$.  This implies that for all $i\in I$ we have
$a_i-a=a-a'_{i},$ and we call this common value $b_i$. Notice that
\be\label{eq:normsequal}
\|b_i\|^2=\la a_i-a,a-a'_{i}\ra=\|a\|^2,
\ee
where we use $\la a_i, a'_{i}\ra=0$ (this follows from $(ii)$) and the definition of $a$. For all $i\in I$  the vectors  $a\pm b_i$ are  entrywise nonnegative which  implies that $|b_i(k)|\le a(k)$ for all $k\in [d]$ and $i\in I$. This fact combined with \eqref{eq:normsequal} implies that $b_i=s_i\circ a,$ for some  $s_i \in \{\pm1\}^d$. { Substituting $a_j=b_j+a$ in $\sum_{j\in J}a_j=\lvert J\rvert a$ it follows that $\sum_{j\in J }b_j=0,$ which in turn  implies that
$\sum_{j\in J} s_j\circ a=0$.  For  $k\in [d]$ with  $a(k)\ne 0$   we get  $\sum_{j\in J}s_i(k)=0$, a contradiction since $s_i \in \{\pm1\}^d$ and  $|J|$ is odd.} {As $\| a \| = \| c \| > 0$ (since $c \neq 0$ by assumption) there must exist a $k$ such that $a(k) \neq 0$.}
\end{proof}
\vspace{0.1cm}

Using Theorem \ref{thm:necessarycondnotcp} we now give a new  family  of matrices in~$\CPSD
\setminus \CP$.
\medskip
\begin{corollary}\label{cor:cpsdsetminuscp}
Let $n=2\ell,$ where {$\ell \geq 3$} is odd.
For  $0\le k\le n-1$ define the Lorentz cone vectors $p_k:=(1,\cos {2\pi k\over n},\sin {2\pi k\over n})$.
Clearly,  we have that
$$(p_k+p_{k+\ell})/2=(1,0,0), \text{ and }   \la p_k,p_{k+\ell}\ra=0,  \text{ for all }\  0\le k \le \ell -1.$$
Furthermore, we have  that $\la p_k,p_{k'}\ra\ge 0$ for all $0\le k,k'\le n-1$. Lastly,
note that
$$\sum_{k=0}^{\ell-1} p_{2k}=\sum_{k=0}^{\ell-1} p_{2k+1}=\ell \cdot (1,0,0).$$
Since $\ell$ is odd, it follows from   Theorem \ref{thm:necessarycondnotcp} that $X:=\gram(\{p_k\}_{k=0}^{n-1})$ is not completely positive. Moreover, as  $\{p_k\}_{k=0}^{n-1}\subseteq \calL_3$  it follows that $X\in \CL^n {\setminus \CP^n}$. In particular we have that $X\in \CPSD^n {\setminus \CP^n}
$.
\end{corollary}

\subsection{Gram-Lorentz behaviors}\label{sec:unbiasedarelorentz}

In view of Theorem \ref{thm:conicformulation}, to any set $\calK\subseteq
\CPSD$ we can associate a family  of quantum   behaviors which we denote by $\calQ_{\calK}$.
We refer to  the quantum behaviors $\calQ_{\CL}$
corresponding to  $\calK=\CL$ as   {\em Gram-Lorentz behaviors}.

As it turns out   Gram-Lorentz behaviors are quite interesting from a physical point of view. First of all, by Theorem \ref{upperboundcpsd} it follows that we can place an upper bound on the size of a quantum system necessary to generate all Gram-Lorentz behaviors, i.e.,
\be\label{eq:gramloretzwhatever}
\max\{ \mathcal{D}(\pbf): \pbf\in \calQ_{\CL}\}<+\infty.
\ee

Note that \eqref{eq:gramloretzwhatever} is in stark contrast to the case of arbitrary quantum behaviors, where no finite  bound is currently known (recall Proposition \ref{eq:proposition} and the discussion preceding it).  In fact, as was already mentioned in the introduction,  the only  quantum behaviors for which we can a piori bound the size of a quantum system necessary to generate them  are the unbiased  behaviors corresponding to a Bell scenario with  binary outcomes \cite{TS87}.   In fact, we can recover this by combining   \eqref{eq:gramloretzwhatever} with  the following result.

\medskip
\begin{theorem}\label{thm:unbiasedquantumislorentz}
In any $(m_A,m_B,2,2)$-scenario, all unbiased quantum behaviors  are Gram-Lorentz behaviors.
\end{theorem}
\medskip

The proof of Theorem \ref{thm:unbiasedquantumislorentz} is deferred to Section \ref{sec:corrtobeha} (cf. Remark \ref{rem:unbiasedgramlorentz}).

A second  interesting fact   is that there exist Gram-Lorentz behaviors for which any quantum representation has  size  exponential in $m_A$ and $m_B$. Specifically, our main result in Section \ref{sec:glbehaviorsstuff} (cf. Theorem \ref{thm:GLbehaviorlowerbound}) is that for any  $n\ge 1$ there exists a Gram-Lorentz behavior $\pbf_n$ corresponding to the $(n,n,2,2)$-scenario satisfying~$\mathcal{D}(\pbf_n)\ge 2^{\Omega(\sqrt{n})}.$

As an immediate  consequence of this fact  we  get that no finite dimension suffices to generate all    behaviors  in $\cup_{n\ge 1} \calQ(n,n,2,2)$.  This  was   the main result  in~\cite{VP09}.

Lastly, the existence  of  Gram-Lorentz  behaviors  for which every quantum representation has exponential size is {our} crucial step for constructing Gram-Lorentz matrices whose $\CPSDR$ is exponential in terms of their size (cf. Section~\ref{sec:highcpsdrank}).

\section{$\CPSD$ matrices whose cpsd-rank is exponential in terms of their size}\label{sec:explowerbounds}

This  section is dedicated to the proof of Result \ref{res:first}, i.e., we show that for   any $n\ge 1$ there  exists a matrix   $X_n\in \CL^{2n}$ such that
$ \CPSDR(X_n)\ge  2^{\Omega(\sqrt{n})}.$ The proof is given  in Section~\ref{sec:highcpsdrank} (cf.  Theorem \ref{thm:lboundgl}) {and} relies on Theorem \ref{thm:minsizequantumcorrelation}.
 Specifically,  given a quantum behavior $\pbf \in \calQ$ it follows   by Theorem \ref{thm:minsizequantumcorrelation}  that $\CPSDR(R)\ge~\mathcal{D}(\pbf),$ for any $  R\in \CPSD^n\cap \mathcal{A}(\pbf)$.
 Consequently, in order to derive   Result \ref{res:first}  it suffices to identify a sequence  of  Gram-Lorentz  behaviors $(\pbf_n)_{n\in \mathbb{N}}$  for which  all quantum representations require a quantum system of size exponential in~$n$.    We show that for every $n\ge 1$ there exists a Gram-Lorentz behavior $\pbf_n$ corresponding to the $(n,n,2,2)$-scenario with the property that  $\mathcal{D}(\pbf_n)\ge 2^{\Omega(\sqrt{n})}.$  This is the main step for showing  Result~\ref{res:first} and its proof is given in
 Section \ref{sec:glbehaviorsstuff} (cf. Theorem \ref{thm:GLbehaviorlowerbound}).  To prove this, instead of working with quantum behaviors we take the equivalent viewpoint of quantum correlations.  This allows us to use  a lower bound  on the size of matrix  representations of extremal quantum correlations, which is implicit in~\cite{TS87}. This is explained in Section~\ref{sec:tsirelson} and Section \ref{sec:dsvefwefwe}   in the~Appendix.
In Section~\ref{sec:corrtobeha}   we show  that to any extremal
quantum correlation $C$  we can associate a Gram-Lorentz behavior
$\pbf_C$  with the property that  $\mathcal{D}(\pbf_C)\ge
2^{\Omega(\rank(C))}$. In Section \ref{sec:extremepoints} we focus
on the  {$(n,n,2,2)$-scenario} and show that any extreme point of
the $n$-dimensional elliptope $\calE_n$  is also an extreme point of
the corresponding set of quantum correlations.
  It is well-known that for any $n\ge1 $ there exist extreme points of $\calE_n$ with rank $\Theta(\sqrt{n})$. Thus,   quantum behaviors corresponding to high-rank extreme points of the elliptope  have the required properties.   Furthermore, in Section \ref{sec:puttingeverythingtogether} we give an explicit family  of Gram-Lorentz behaviors achieving the exponential lower bound. We conclude the proof of Result \ref{res:first} in Section~\ref{sec:highcpsdrank} and give an explicit family of matrices with exponentially large~cpsd-rank.}

\subsection{Quantum  correlations}\label{sec:tsirelson}

Throughout this section, for notational convenience we
set $n:=m_A$ and $m:=m_B$. Furthermore, we   focus on the $(n,m,2,2)$-scenario and we assume that the measurement outcomes are  given by  $\{\pm 1\}$. We denote by $\calQ$ the corresponding set of quantum behaviors.

{For the reader's convenience  we have collected   in this
subsection some facts we use in later parts of this work. These
results are  well-known in the quantum information community but
much less so {in} the mathematical optimization community. }

We first describe a well-known equivalent  parametrization of the set of quantum behaviors, which is the  appropriate  language for stating
 Tsirelson's theorem (cf. Theorem \ref{thm:tsirelson1}). To describe  this we use the map  $f :\R^{4nm}  \rightarrow  \R^{n+m+nm},$ which maps the  behavior $\pbf=\left(p(ab|xy)\right)$  to the vector  $\mathbf{c}=\left( c_x, c_y, c_{xy}\right)$ where
  \be\label{eq:expectations}
   c_x:=\sum_{a\in \{\pm 1\}}a\ p_A(a|x), \  c_y:=\sum_{b\in \{\pm 1\}} b\  p_B(b|y), \text{ and }  c_{xy}:=\sum_{a,b\in \{\pm 1\}}ab\ p(ab|xy).
\ee

Note that $c_{xy}$ corresponds to the expected value  of the product of the players' outcomes, given that they performed measurements $x$ and $y$, respectively.  Similarly, $c_x$  and  $c_y$ correspond to the  expected values of the player's individual outcomes.

The map $f$  is   linear and injective. Consequently,
 the set $\calQ$  of quantum behaviors  is  in one-to-one correspondence with
$f(\calQ)$, i.e.,  the image of $\calQ$    via the map~$f$. We refer to $f(\calQ)$  as the set  of {\em full quantum  correlations}. A  full quantum correlation $ \left( c_x, c_y, c_{xy}\right)$ is called {\em unbiased} if $c_x=c_y=0$, for all $x,y$. Lastly, note that
the inverse of $f$ is the map $g:  \R^{n+m+nm} \rightarrow \R^{4nm}, $ which maps a full quantum correlation  $\mathbf{c}=\left( c_x, c_y, c_{xy}\right)$   to  the {behavior} $\pbf:=g(\cbf)$  defined as
\be\label{eq:expectationstovectors}
p(ab|xy)={1 + a \, c_x + b \, c_y + ab \, c_{xy} \over 4}.
\ee

{The following  lemma gives a characterization  of the set of full quantum correlations.  We have included a  proof for completeness, which we also use in  Remark~\ref{rem:equivalencebetweenrepresentations}.}

\medskip 
\begin{lemma}\label{lem:qrealizationcorr}
The vector  $\mathbf{c}=\left( c_x, c_y, c_{xy}\right)\in [-1,1]^{n+m+nm}$ is a full quantum correlation
if and only if  there exist Hermitian  operators
 $\{M_x\}_{x}, \{N_y\}_{y}$ with eigenvalues in $[-1,1]$ and a quantum state $\rho$ such that, {for all x, y, we have}
\be
\label{eq:qrealization}
c_x=\tr((M_x\otimes I)\rho), \, c_y=\tr((I\otimes N_y)\rho), \,  \text{ and } \, c_{xy}=\tr((M_x\otimes N_y)\rho).
\ee
\end{lemma}
\medskip 

\begin{proof}
Consider $\pbf\in \calQ$ such that $f(\pbf)=\cbf$ and let $\left\{ \{
M_{a|x}\}_a, \{ N_{b|y}\}_b, \rho\right\}$ be a quantum representation  for $\pbf$.
For $x\in [n]$  set $M_x:=M_{1|x}-M_{-1|x}$ and for $y\in[m]$ set $N_y:=N_{1|y}-N_{-1|y}$. Since $I=M_{1|x}+M_{-1|x}=N_{1|y}+N_{-1|y},$ for all $x,y$ it follows that $M_x$ and $N_y$ have eigenvalues in $[-1,1]$. Lastly, using  \eqref{eq:expectations}, an easy calculation shows that \eqref{eq:qrealization} is satisfied.

Conversely,  let $\{M_x\}_{x}, \{N_y\}_{y}$ be Hermitian operators with eigenvalues in $[-1,1]$
 and $\rho$  a quantum state satisfying \eqref{eq:qrealization}. For $x\in [n]$ and $a\in \{\pm 1\}$ set $M_{a|x}={I+aM_x\over 2}$ and similarly,  for $y\in [m]$ and $b\in \{\pm 1\}$ set $N_{b|y}={I+bN_y\over 2}$. Note that $\{M_{a|x}\}_{a}$ and $\{N_{b|y}\}_b$ are valid POVMs.
 Lastly, defining  the quantum behavior $\pbf $ where  $p(ab|xy)=\tr((M_{a|x}\otimes N_{b|y})\rho),$ it follows {that  $\mathbf{c}=f(\pbf)$ and is thus a full quantum correlation}.
\end{proof}
\medskip

Given a full quantum correlation $  \cbf=\left( c_x, c_y, c_{xy}\right)$ we  refer to any   ensemble of Hermitian operators  $\left\{ \{M_x\}_{x}, \{N_y\}_{y}, \rho\right\}$  as defined in Lemma \ref{lem:qrealizationcorr} as a {\em quantum representation} of $\cbf$. We say that a  quantum  representation  of $\cbf$ is {\em $d$-dimensional} if $\{M_x\}_x,\{N_y\}_y {\subseteq \calh^d}$ and {$\rho \in \calh_+^{d^2}$}.

\medskip 
\begin{remark}\label{rem:equivalencebetweenrepresentations}
{From the proof of Lemma~\ref{lem:qrealizationcorr}, we see we have that $\pbf\in \calQ$ has a $d$-dimensional quantum representation (as a bevavior) if and only if $f(\pbf)$ has a $d$-dimensional quantum representation (as a full quantum correlation).}
\end{remark}
\medskip

We denote by $\cornm$  the   projection of the set of full quantum
correlations onto $\R^{nm}$, {that is, {we only keep}  the entries
$(c_{xy})_{xy}$,} and refer to its elements as  {\em quantum
correlations}. It is  sometimes  useful   to arrange the entries of
a quantum correlation $\cbf \in \cornm$ as a {matrix $C$} in
$[-1,1]^{n\times m}$, in which case we write $C\in \cornm$.
Throughout this section  we use these two  forms interchangeably.

Tsirelson's theorem~\cite{TS87} given below  has two important   consequences: First, it  characterizes the set  of quantum correlations as the feasible region of a semidefinite  program (cf. condition $(iii)$). {Second},  condition $(ii)$ implies
that all unbiased quantum behaviors can be generated using quantum systems of finite dimension.

\medskip 

\begin{theorem}[\cite{TS87}]\label{thm:tsirelson1} For any  $C=(c_{xy})\in [-1,1]^{n\times m}$ the  following are equivalent:
\vspace{0.01cm}
\bi

\item[$(i)$] {$C$ is a quantum correlation, i.e.,} there exist Hermitian operators
 $\{M_x\}_{x}$, $\{N_y\}_{y}$ with eigenvalues in $[-1,1]$ and  a quantum state $\rho$~satisfying
$$c_{xy}=\tr((M_x\otimes N_y)\rho), \; \text{ for all } \; x\in [n], y\in [m].$$

\item[$(ii)$] There exist unit vectors $\{u_x\}_x$ and $\{v_y\}_{y}$ in $\R^{n+m}$ such that
\vspace{0.1cm}

\bi \item[$(a)$] $c_{xy}=\Psi_d^*(\gamma(u_x)\otimes
\gamma(v_y)^\top)\Psi_d,\ $ for all  $x\in [n],y\in [m]$;
\vspace{0.1cm}
\item[$(b)$]  $\Psi_d^*( \gamma(u_x)\otimes I)\Psi_d = 0,\ $ for all  $x\in [n]$;
\vspace{0.1cm}
\item[$(c)$] $\Psi_d^*(I \otimes \gamma(v_y)^\top)\Psi_d = 0,\ $ for all  $y\in
[m]$, \ei where $d:=2^{\left\lfloor \frac{n+m}{2}\right\rfloor}$, $\Psi_d$ is the $d$-dimensional maximally
entangled state defined in \eqref{eq:maxentangled}  and %Section~\ref{sec:preliminaries},
the map $\gamma$ is defined in Theorem \ref{thm:cliffordalgebramain}.
 \vspace{0.1cm}
\item[$(iii)$] There exist unit vectors $\{u_x\}_x$ and $\{v_y\}_{y}$ in $\R^{n+m}$ such that
$$c_{xy}=\la u_x,v_y\ra, \; \text{ for all } \; x\in [n], y\in [m].$$
\ei
\end{theorem}

The next result, which  is implicit  in  Tsirelson's work
\cite{TS87}, gives a lower bound on the size of a quantum
representation for any extreme point of the set of quantum
correlations. Since this is not stated explicitly in  \cite{TS87},
for completeness  we have included a short  proof in 
{Appendix} \ref{sec:dsvefwefwe}.

\medskip 
\begin{theorem}[\cite{TS87}]\label{thm:mainlowerbound1}
 Let   $C=(c_{xy}) \in {{\rm ext}(\cornm)}$  and consider a  family of   Hermitian operators
 $\{M_x\}_{x}, \{N_y\}_{y}\subseteq \calh^d$ with eigenvalues  in $[-1,1]$ and a quantum state $\rho\in\calh^{d^2}_+$ satisfying
$c_{xy}=\tr((M_x\otimes N_y)\rho),$ for all $x,y$. Then we have that
$$d\ge {\sqrt{2}^{\lfloor\rank(C)/ 2\rfloor}}.$$
\end{theorem}
\medskip

{We note that Slofstra~\cite{S11} generalized Tsirelson's lower bound given above  by considering {\emph{near-extremal}} quantum correlations and their {\emph{approximate representations}}.}

As we explain in the next section, Theorem
\ref{thm:mainlowerbound1} turns out to be the  main ingredient for
constructing  cpsd matrices whose cpsd-rank is exponential in terms
of their {sizes}.

\subsection{Gram-Lorentz behaviors with large quantum representations}\label{sec:glbehaviorsstuff}

{In this section  we show that for every $n\ge 1$ there exists a Gram-Lorentz behavior $\pbf_n$ corresponding to the $(n,n,2,2)$-scenario such that  $\mathcal{D}(\pbf_n)\ge 2^{\Omega(\sqrt{n})}$ (cf. Theorem~\ref{thm:GLbehaviorlowerbound}). }

\subsubsection{Going from quantum correlations to {Gram-Lorentz} behaviors}\label{sec:corrtobeha}

By Theorem~\ref{thm:tsirelson1} we can associate a quantum behavior to any  quantum  correlation.

\medskip 
\begin{definition}\label{def:goodbehaviors}
For any  {$C \in \cornm$} we denote by  $\pbf_C=(p_C(ab|xy))$ the quantum behavior  given by {$g((0,0,C))$}.  Concretely, by  \eqref{eq:expectationstovectors}  we have   that
\be\label{eq:inversec}
p_C(ab|xy)={1 + ab \, c_{xy} \over 4}, \text{ for all }a,b,x,y.
\ee
\end{definition}
\medskip

It is also useful to arrange the entries of $\pbf_C$ into a $2n\times 2m$ matrix given by
\be\label{eq:behaviormatrix}
P_C:=\sum_{a,b\in \{\pm 1\} ,x,y\in [n]}p_C(ab|xy)\  e_ae_b^\sfT\otimes e_xe_y^\sfT={1\over 4}\begin{pmatrix}J+C & J-C\\J-C& J+C\end{pmatrix}.
\ee

\medskip
\begin{remark} {Note that the behavior $ \pbf_C$ is well-defined. This follows by Theorem~\ref{thm:tsirelson1} $(ii)$ as $(0,0,C)$ is a full quantum  correlation vector for  any $C \in \cornm$.  }
\end{remark}
\medskip

{As it turns out,  behaviors constructed in this manner have interesting~properties.}

\medskip
\begin{lemma}\label{cddfsvdfv}
For any  $C=(c_{x,y})\in \cornm$ the  behavior $\pbf_C$ is Gram-Lorentz. In particular, consider  unit vectors $\{u_x\}_x$ and $\{v_y\}_{y}$ in $\R^{n+m}$ such that $c_{xy}=\la u_x,v_y\ra,$ for all  $x,y$ (these exist by Theorem \ref{thm:tsirelson1} $(iii)$).  Then we have that
\be
p_C(ab|xy)=\la \ell^x_a, \tilde{\ell}^y_b\ra, \text{ for all } a,b,x,y, \text{ where}
 \ee
 \be\label{eq:lorentzvectors}
\ell^x_a={1\over 2}(1,au_x), \, \forall x\in X, a\in \{\pm1\}, \, \text{ and } \, \tilde{\ell}^y_b={1\over 2}(1, bv_y), \, \forall y\in Y,  b\in \{\pm 1\}.
\ee
\end{lemma}

\begin{proof}
By \eqref{eq:inversec}  we have  that  $p_C(ab|xy)=(1+ab \, c_{xy})/4,$ for all $a,b,x,y$.
By Theorem \ref{thm:cliffordalgebramain} $(iii)$ we get
$c_{xy}=\la u_x,v_y\ra=\Psi_d^*\left(\gamma(u_x)\otimes \gamma(v_y)^\top\right)\Psi_d,$ for all $x,y$, where $d:=2^{\left\lfloor \frac{n+m}{2}\right\rfloor}$.
This gives
\be\label{fevrbregertg}
p_C(ab|xy)=\Psi_d^*\left({I+a\gamma(u_x)\over 2}\right)\otimes\left( {I+b\gamma(v_y)^\sfT\over 2}\right)\Psi_d, \text{ for all } a,b,x,y.
\ee
Set  $$\Gamma(\ell^x_a)=
{1\over \sqrt{d}}\left(\frac{I +a\ \gamma\left(u_x\right)}{2} \right) \in \calh^d_+, \text{ for }   a\in \{\pm 1\},
$$ and
$$\Gamma(\tilde{\ell}^y_b)={1\over \sqrt{d}}\left(\frac{I +b \ \gamma\left(v_y\right)}{2} \right)\in \calh^d_+, \text{ for }   b\in \{\pm 1\},$$
where $\Gamma$ was defined in  \eqref{eq:isometry}. Using \eqref{eq:maxent}, it follows by \eqref{fevrbregertg} that
\be
p_C(ab|xy)= {\la \Gamma(\ell^x_a), \Gamma(\tilde{\ell}^y_b) \ra}=\la \ell^x_a, {\tilde{\ell}_b^y} \ra, \text{ for all } a,b,x,y,
\ee
where we used the fact that $\Gamma$ is an isometry.
Since the vectors $\{u_x\}_x$ and $\{v_y\}_{y}$ are unit it follows  that   the vectors {$\{\ell^x_{a}\}_{a,x}$ ,$\{\tilde{\ell}^y_{b}\}_{b,y}$} belong to the Lorentz cone
$\calL_{m+n+1}$. Furthermore,  by \eqref{eq:lorentzvectors} we have  that  $\ell^x_1+\ell^x_{-1}=\tilde{\ell}^y_1+\tilde{\ell}^y_{-1}=e_1,$ for all $x,y$ {implying $\gram(\{\ell^x_{a}\}_{a,x} ,\{\tilde{\ell}^y_{b}\}_{b,y}) \in \calA(\pbf_C)$}.
Thus the behavior $\pbf_C$ is Gram-Lorentz.
\end{proof}

\medskip

\begin{remark}\label{rem:unbiasedgramlorentz}
As an immediate consequence of  Lemma \ref{cddfsvdfv} it follows that every unbiased quantum behavior is Gram-Lorentz.
\end{remark}

\medskip
We are now ready to   translate  Theorem \ref{thm:mainlowerbound1} to Gram-Lorentz behaviors.

\medskip
\begin{theorem}\label{cor:lowerboundcorrelations}
For any {$C\in {\rm ext}(\cornm)$} we have that $\pbf_C$ is Gram-Lorentz~and
$$\mathcal{D}(\pbf_{C})\ge {\sqrt{2}^{\lfloor \rank(C)/ 2\rfloor}}.$$
\end{theorem}

\begin{proof}
{Fix $C\in {\rm ext}(\cornm)$} and let {$\pbf_C=g((0,0,C))$}.
We already determined  in Lemma \ref{cddfsvdfv} that $\pbf_C$ is Gram-Lorentz.
{By definition,
we have that $\mathcal{D}(\pbf_{C})$ is equal to the {least} integer $d\ge 1$ for which $\pbf_C$ admits  a $d$-dimensional representation. Since $(0,0,C) = f(\pbf_C)$, by Remark~\ref{rem:equivalencebetweenrepresentations} we know that $\mathcal{D}(\pbf_{C})$ is also equal to the {least} integer $d\ge 1$ for which $(0,0,C)$ admits  a $d$-dimensional representation.
By Theorem~\ref{thm:mainlowerbound1}, the latter quantity is lower bounded by {$\sqrt{2}^{\lfloor \rank(C)/ 2 \rfloor}$} as desired.}
\end{proof}
\medskip

In view of Theorem  \ref{cor:lowerboundcorrelations},  to construct Gram-Lorentz behaviors all of whose quantum representations require exponential size, it suffices to identify high-rank extreme points of $\cornm.$ In the next section we consider this problem {for} the case~$n=m$.

\subsubsection{High-rank extremal quantum correlations}
\label{sec:extremepoints}

Throughout this section we set  $n=m$ and we  view   any   $C\in {\rm Cor}(n,n)$ as  a square $ n\times n$ matrix.

Of special interest to us  are the  elements of ${\rm Cor}(n,n)$ whose diagonal entries are all equal to 1.
Specifically, in our next lemma below we show they coincide with the  $n$-dimensional {\em elliptope}, denoted by $\calE_n$, which is defined as the set of $n\times n$ symmetric psd matrices with  diagonal entries   equal to~1. The elliptope is a spectrahedral set  whose structure has  been  extensively  studied (e.g. see \cite{DL} and references therein).

We begin this {section} by determining  a useful relation
between $\calE_n$ and ${\rm Cor}(n,n)$.

\medskip 
\begin{proposition}\label{efweferfer}
{We have that ${\ext (\calE_n) \subseteq \ext ({\rm Cor}(n,n))}.$}
\end{proposition}

\begin{proof}
Fix $X\in \ext (\calE_n)$ and let $X=\lambda A+(1-\lambda)B$, where $A,B\in {\rm Cor}(n,n)$ and $\lambda\in [0,1]$. For all $i\in [n]$ we have that $1=\lambda A_{ii}+(1-\lambda)B_{ii}$  and since  $A_{ii},B_{ii}\in [-1,1]$ it follows that
$1=A_{ii}=B_{ii}$, for all $i\in [n].$

We now show that  $A,B\in \calE_n$, and the proof is concluded by the extremality assumption.  By Theorem \ref{thm:tsirelson1} $(iii)$
  there exist unit vectors $\{u_i\}_i$ and $\{v_j\}_j$ such that $A_{ij}=\la u_i,v_j\ra,$ for all $i,j\in [n]$. By the Cauchy-Schwartz inequality we have that
$1=A_{ii}=\la u_i,v_i\ra \le 1, $ for all $i\in [n]$.  Thus, equality holds throughout which implies that $u_i$ is parallel to $v_i,$ for all $i\in[n]$. Lastly, as all vectors have unit norm  we have   $u_i=v_i$ for all $i\in [n]$ {and thus $A\in \calE_n$}. Similarly, we have $B\in \calE_n$.
\end{proof}
\medskip

Using  Proposition  \ref{efweferfer} we can construct extreme points
of ${\rm Cor}(n,n)$ using extreme points of $\calE_n$. This is
extremely  useful as the extreme points of the elliptope are
completely understood.  We explain this in the remaining part of
this section.

Throughout, we denote by   $r_{\max}(n)$  the {greatest} integer
satisfying $\binom{r+1}{2}\le n$,~i.e.,
$$r_{\max}(n)=\left\lfloor {\sqrt{1+8n}-1 \over 2}\right\rfloor.$$

We now state two well-known results concerning properties of extreme
points of the elliptope that we use in the next section.   The first
one due to  \cite{LT94} (see also \cite[Corollary 31.5.4]{DL})
allows {one} to easily check whether a matrix $X\in \calE_n$   is
{an extreme~point}.

\medskip
\begin{theorem} [{\cite{LT94}}] \label{thm:perturbationsdimension}
Let $X\in \calE_n$  with $\rank(X)= r$ and let   $\{u_i\}_{i=1}^n\in \R^r$  be a Gram representation of $X$.
Then $X\in \ext {(\calE_n})$ if and only if
$$\dim ( \Span( { \{ u_iu_i^\sfT: i\in [n] \}}))=\binom{r+1}{2}.$$
\end{theorem}
\medskip

The second result due to \cite{GPW} (see also \cite[Proposition
31.5.7]{DL}) specifies the range of possible ranks for the  extreme
points of the elliptope and moreover shows that every value  in that
range is {achievable}.

\medskip
\begin{theorem}[{\cite{GPW}}]\label{thm:extremepoints}
For any $X\in \ext (\calE_n)$   we have that $\rank(X)\le r_{\max}(n)$. Furthermore, for any integer $r$ in the  range  $1\le r\le r_{\max}(n)$ there exists $X_r\in \ext (\calE_n)$ with $r=\rank(X_r)$.
\end{theorem}
\medskip

\begin{example}\label{ex:extremepointconstruction}
We now describe the constructive part of Theorem \ref{thm:extremepoints} which we use in the next section.
Fix an integer  $r$ satisfying $1\le r\le r_{\max}(n)$.   In particular we have that $\binom{r+1}{2} < n+1$. Let $\{e_i\}_{i=1}^r$ be the standard basis in $\R^r$. For $i,j\in [r]$ define $w_{i,j}:={1\over \sqrt{2}}(e_i+e_j)$.  Define  $X_r$ to  be the Gram matrix of the following  family of vectors:   we use $e_1$  repeated $n+1- \binom{r+1}{2}$ times,  followed by  $e_2,\ldots, e_r$ one time each and lastly,  we use $w_{ij}$ for all $ 1\le i<j\le [r]$. Clearly,  we have that $\rank(X_r)=r$. Furthermore, since the matrices $e_ie_i^\sfT,w_{ij}w_{ij}^\sfT$ are linearly independent it follows that
$$\dim \left( \Span\left(\{e_ie_i^\sfT\}_{i=1}^r, \{ w_{ij}w_{ij}^\sfT\}_{1\le i<j\le [r]} \right) \right)=r+\binom{r}{2}=\binom{r+1}{2}.$$
By Theorem \ref{thm:perturbationsdimension} it follows that  $X_r\in \ext (\calE_n)$.
\end{example}

\subsubsection{Putting everything together}\label{sec:puttingeverythingtogether}

Combining the results given in Sections~\ref{sec:corrtobeha} and \ref{sec:extremepoints} we   now  show that for every $n\ge 1$ there exists a Gram-Lorentz behavior $\pbf_n$ corresponding to the $(n,n,2,2)$-scenario satisfying~$\mathcal{D}(\pbf_n)\ge 2^{\Omega(\sqrt{n})}.$

\begin{theorem}\label{thm:GLbehaviorlowerbound}
Fix $n\ge 1$  and let $C_n\in \ext (\calE_n)$ with  $\rank(C_n)=r_{\max}(n)$. Then
$$\mathcal{D}(\pbf_{C_n})\ge  {\sqrt{2}^{\lfloor \rmax(n)/ 2\rfloor}}.$$
\end{theorem}

\begin{proof}
By Theorem \ref{thm:extremepoints} there exists  $C_n\in \ext (\calE_n)$ with  $\rank(C_n)=r_{\max}(n)$.  By~\eqref{efweferfer}  it follows that $C_n\in \ext  ({\rm Cor}(n,n))$. The proof is concluded by Theorem~\ref{cor:lowerboundcorrelations}.
\end{proof}
\medskip

We conclude this section with an explicit example. To ease the exposition   we only consider matrices of size $N:=2n^2+n,$ for  any $n\ge 1$. In this case  $\rmax(N)=2n$.

By Theorem \ref{thm:extremepoints} there exists $C_n\in \ext (\calE_N)$ with $\rank(C_n)=2n$. As described in
Example~\ref{ex:extremepointconstruction}, the matrix $C_n$  is defined as the Gram matrix of the  vectors
\be\label{cewverger}
w_{ii}:=e_i, \text{ for } i\in [2n] \; \text{ and } \; w_{ij}:={1\over \sqrt{2}}(e_i+e_j), \text{ for } 1\le i<j\le [2n].
\ee

It is instructive to think of the underlying Bell scenario  as each
player having $\binom{2n+1}{2}$ questions that are  indexed by the
2-element {\em multisets}  of  $[2n]$.  In particular,  the first
$2n$ questions  correspond to the  {multisets} $\{\{i,i\}: i\in
[2n]\}$ and the remaining $\binom{2n}{2}$ questions  correspond to
$\{\{i,j\}: 1\le i<j\le 2n\}$.

By construction, the entries of $C_n$ are given by
\be C_n=\begin{pmatrix} I_n & A_n\\A_n^\sfT & B_n\end{pmatrix}, \text{where}
\ee
\be\label{eq:goodexample}
 A_n[ii,kl]=\begin{cases}{1\over \sqrt{2}}, \text{ if } i\in \{k,l\},\\
0, \text{ otherwise},
\end{cases}  \text{ and } \ B_n[ij,kl]={1\over 2}|\{i,j\}\cap \{k,l\}|.
\ee
Lastly, using \eqref{eq:behaviormatrix} we have that
\be\label{ex:explicitexample}
P_{C_n}={1\over 4}\begin{pmatrix} J+C_n& J-C_n\\ J-C_n & J+C_n\end{pmatrix}.
\ee

\subsection{cpsd-matrices with high cpsd-rank}\label{sec:highcpsdrank}

In this  section we give the proof of  Result \ref{res:first}, i.e.,  we show that for any $n\ge 1$ there exists   $X_n\in \CL^{2n}$  such that $\CPSDR(X_n)\ge~2^{\Omega(\sqrt{n})}.$ This follows by combining   Theorem \ref{thm:GLbehaviorlowerbound}  with Theorem \ref{thm:minsizequantumcorrelation}.

\medskip
\begin{theorem}\label{thm:lboundgl}
Fix  $n\ge 1$ and let  $C_n\in \ext (\calE_n)$ with  $\rank(C_n)=r_{\max}(n)$. Then
\be\label{eq:goodglmatrix}
P_{C_n}:=
{1\over 4} \begin{pmatrix}J+C_n & J-C_n\\J-C_n& J+C_n\end{pmatrix}
\ee
is a  $2n\times 2n$ Gram-Lorentz matrix
satisfying
 \be\label{eq:cedfergrth}
\CPSDR(P_{C_n})\ge  {\sqrt{2}^{\lfloor \rmax(n)/ 2\rfloor}}.
\ee
\end{theorem}
\medskip

\begin{proof}
By Lemma \ref{cddfsvdfv} we get  that ${P_{C_n}\in \CL^{2n}}$.
Furthermore, as   $P_{C_n}\in\CL^{2n}\subseteq  \CPSD^{2n}$ we
have that $X_n := \left(\begin{smallmatrix}P_{C_n} & P_{C_n}\\P_{C_n} & P_{C_n}\end{smallmatrix}\right)\in {\CPSD^{4n}},$ {since the psd matrices in the $\CPSD$-factorization can be repeated}. {Also, we clearly have that $X_n \in \calA(\pbf_{C_n})$. Thus, by Theorem \ref{thm:minsizequantumcorrelation}}, we get   ${\CPSDR(X_n) \geq \mathcal{D}(\pbf_{C_n})}$.  {It is easy to verify that $\CPSDR(X_n) = \CPSDR(P_{C_n})$}. Lastly, since  $C_n\in \ext ({\rm Cor}(n,n))$, by Theorem \ref{thm:GLbehaviorlowerbound} we have $\mathcal{D}(\pbf_{C_n})\geq {\sqrt{2}^{\lfloor \rmax(n)/ 2\rfloor}}$ and the proof is concluded.
\end{proof}

\begin{remark}\label{rem:tightlowerbound}
In Theorem
\ref{upperboundcpsd} we determined that for any $X\in \CL^n$ we have that
$ \CPSDR(X)\le 2^{\lfloor (\rank(X)+1)/2\rfloor}.$ Since {$\rank(P_{C_n})\le \rank(C_n) + 1  $}, this upper bound applied to the matrices $P_{C_n}$ defined in  \eqref{eq:goodglmatrix} shows that  for all $n\ge 1$ we have $  \CPSDR(P_{C_n})\le {2^{\lfloor (\rmax(n)+2)/2\rfloor} = 2^{\lfloor \rmax(n)/2\rfloor + 1}}$.
Thus, in view of \eqref{eq:cedfergrth}, the upper bound on the $\CPSDR$ of $\CL$ matrices given in Theorem \ref{upperboundcpsd} is essentially tight.
\end{remark}
\medskip

Returning  to the example \eqref{ex:explicitexample} from  Section \ref{sec:puttingeverythingtogether} it follows that  $P_{C_n}\in~\CL^{2(2n^2+n)}$ and   $\CPSDR(P_{C_n})\ge {\sqrt{2}^{\lfloor (2n-1)/2 \rfloor}}$.
In particular, Lemma \ref{cddfsvdfv} implies that  the vectors
$$\ell^{ij}_a:=\left({1\over 2},{a \, w_{ij}\over 2}\right), \ a\in \{\pm 1\}, \ 1\le i\le j\le 2n,$$
lie in $\calL_{2n+1}$ and  give a $\CL$-factorization  of $P_{C_n}$ (for the definition of the $w_{ij}$'s see \eqref{cewverger}).  The corresponding $\CPSD$-factorization is given by  the
{psd} matrices $\{\Gamma(\ell^{ij}_a)\}_{ij,a},$~{where}
$$\Gamma{(\ell^{ij}_a)}={1\over 2^{n/2}}\left ({I + a \, \gamma(w_{ij})\over 2}\right)\in \calh_+^{2^n}.$$

\section{cpsd-graphs}\label{sec:cpsdgraphs}

We say that  $G=([n],E)
$  is  a  {\em cpsd-graph}   if for any matrix $X\in~\DNN^n$ whose support is given by $G$, i.e.,  $ S(X)=G$, we have  that  $ X\in \CPSD^n.$
The analogous notion of   {\em cp-graphs}  has been   studied extensively (e.g. see \cite[Section~2.5]{CP}).
In fact, the class of cp-graphs admits an exact characterization: A graph is cp  if and only if  it does not contain an odd cycle $C_{2t+1}$ ($t\ge 2$)  as a subgraph~\cite{KB}. In this section  we show that  the same characterization extends to  {cpsd-graphs} (cf. Theorem~\ref{thm:cpsdgraphs}).

{To arrive at the characterization of cpsd-graphs, we generalize a sufficient condition from \cite{FW} for constructing doubly-nonnegative matrices that are not
cpsd.  As noted in \cite{LP14}, the example of the matrix  in  $\DNN^5\setminus \CPSD^5$ given in \cite{FW} does not admit a Gram factorization  by  positive elements in any finite  von Neumann algebra.  Our   sufficient condition given in Theorem~\ref{thm:dnnminuscpsd} below  generalizes this construction.}

\subsection{$\DNN$ matrices with no  $\mathcal{N}^+$-factorizations}\label{thm:neccondition1}

First, we introduce    some necessary
background   on  von Neumann algebras. We  keep the discussion to a minimum and
 {refer the} interested reader to \cite{Murphy} for a
comprehensive introduction.

A {\em von Neumann algebra}  is a unital  $\ast$-subalgebra of the $C^\ast$-algebra  of bounded operators on a Hilbert space $H$,  that is closed in the weak operator topology.
A von Neumann algebra $\mathcal{N}$  is called {\em tracial}  if it is equipped with
 a linear functional  $\tau: \mathcal{N}\rightarrow~\C$  satisfying:  $(i)$ $\tau(x^*x)\ge0$ for all $x\in \mathcal{N}$ and $\tau(1)=1$ $ (ii)$ $\tau(x^*x)=0\Longrightarrow x=0$ $(iii)$ $\tau(xy)=
\tau(yx),$ for all $x,y\in \mathcal{N}$ and $(iv)$ the restriction of $\tau$ to the unit ball is continuous with respect to the weak operator~topology.

An element $p\in \mathcal{N}$ {is} called {\em positive} if $p=x^\ast x,
$ for some $x\in \mathcal{N}$.   We denote by $\mathcal{N}^+$ the set
of positive elements in $\mathcal{N}$. We  make use of  the  fact
that any  $p\in \mathcal{N}^+$ has  a unique positive  square root
(e.g. see  \cite[Theorem 2.2.1]{Murphy}).

\medskip
\begin{remark}\label{rem:kernelvectorvn}Let $(\mathcal{N},\tau)$ be a tracial  von Neumann algebra.
 Let   $\{x_i\}_{i=1}^n\subseteq~\mathcal{N}$ such that $x_i^\ast=x_i$ for all $i\in [n]$ and set    $X:=(\tau(x_ix_j)_{1\le i,j\le n})$. For any  $u\in {\rm Ker} X$ we have  that $\sum_{i=1}^nx_i u_i=0$. Indeed, note that
\[ 0=u^\ast Xu=\sum_{i,j=1}^n\bar{u}_iu_j\tau(x_ix_j)=\tau {\left( \left( \sum_{i=1}^nu_ix_i \right)^\ast \left( \sum_{i=1}^nu_ix_i \right) \right)}, \]
which by $(ii)$  implies that $\sum_{i=1}^nu_ix_i=0$.

Moreover, if $\tau(pq)=0$ where $p,q$ are positive elements of $\mathcal{N}$ then we have that  $pq=0$. To see this let $p=a^\ast a$ and $q=b^\ast b$ and note that $\tau(pq)=\tau(a^\ast ab^\ast b)=\tau((ab^\ast)^\ast ab^\ast)=0$ which by $(ii)$ implies that  $ ab^\ast=0$.  This shows that $pq=0$.
\end{remark}
\medskip

Let   $(\mathcal{N},\tau)$ be a tracial  von Neumann algebra. We say
that a matrix  $X\in \DNN^n$ admits an {\em
$\mathcal{N}^+$-factorization} if   there exist positive elements
$\{p_i\}_{i=1}^n\subseteq \mathcal{N}^+$ such that
$X=(\tau(p_ip_j)_{1\le i \le j\le n}).$ Next we  give a sufficient
condition for constructing $\DNN$ matrices for which no
$\mathcal{N}^+$-factorization exists, generalizing a construction
from~\cite{FW}.

\medskip
\begin{theorem}\label{thm:dnnminuscpsd}
Consider  nonzero
vectors $\{u_i\}_{i=1}^n\subseteq \R^d$ such that $\la u_i,u_j\ra\ge 0 $ for all $i,j\in [n].$ Assume that there exist  subsets $ I, J\subseteq  [n]$ with the following properties:
\begin{itemize}
\item[$(i)$]  $\Span({\{u_i: i \in I \}})=\Span({\{u_j: j\in J \}})= \Span({\{ u_i: i\in [n] \}});$
\item[$(ii)$] There exists $i^*\in I $ such that $\la u_{i^*}, u_i\ra =0, $ for all $i\in I\setminus \{i^*\}$;
\item[$(iii)$] There exists $j^*\in J $ such that $\la u_{j^*}, u_j\ra =0,$ for all $j\in J\setminus \{j^*\}$;
\item[$(iv)$]  The vector $u_{i^*}$ is not parallel to $ u_{j^*}$;
\item[$(v)$]  We have $\la u_{i^*},u_{j^*}\ra\ne 0$.
\end{itemize}
Then the matrix $\gram(\{u_i\}_{i=1}^n)$ does not admit an $\mathcal{N}^+$-factorization  for any tracial  von Neumann algebra $(\mathcal{N},\tau)$.
\end{theorem}
\medskip

\begin{proof}
Let $(\mathcal{N},\tau)$ be a tracial  von Neumann algebra  and    let $\gram(\{u_i\}_{i=1}^n)=(\tau(p_ip_j)_{1\le i
\le j\le n}),$ for some positive elements $\{p_i\}_{i=1}^n\subseteq
\mathcal{N}$.  By~$(i)$  we have that $u_{i^*}\in \Span( \{u_j: j\in
J\}) $ so   Remark \ref{rem:kernelvectorvn} implies    that
$p_{i^*}\in \Span( \{p_j: j\in J\}). $ Pre-multiplying  this by
$p_{j^*}$, it follows from  $(iii)$ that $p_{j^*}p_{i^*}\in \Span(
\{p_{j^*}^2\})$, {where we have utilized the fact that
$\tau(p_{j^*}p_{i^*})=0$ implies $p_{j^*}p_{i^*}=0$}.
Analogously,   $(i)$ implies  that $p_{j^*}\in \Span(\{p_i: i\in I\}) $ and post-multiplying   by~$p_{i^*}$ we get from $(ii)$ that $p_{j^*}p_{i^*}\in \Span( \{ p_{i^*}^2\})$.
By $(v)$ we get  $p_{j^*}p_{i^*}\ne 0$ and  combining the two equations,
there exists a  scalar {$c \neq 0$} such that $ p_{i^*}^2=cp_{j^*}^2$.
 {Also note that $c > 0$ since
 $0 < \tau(p_{i^*}^2) = c \tau (p_{j^*}^2)$ and $\tau (p_{j^*}^2) > 0$.}
 Since each positive element of a $C^\ast$-algebra has a unique positive {square} root  we have  $p_{i^*}=\sqrt{c}p_{j^*}.$ This  contradicts~$(iv)$.
 \end{proof}
\medskip

Based on  Theorem \ref{thm:dnnminuscpsd},  we now give  a family of $\DNN$ matrices supported by $C_{2t+1}$ (for all $t\ge 1$) that  do not  admit a Gram factorization  with positive elements in any   tracial  von Neumann algebra.

\medskip
\begin{lemma}\label{lem:oddcycles}
Let $A_t$ denote  the adjacency matrix of $C_{2t+1}$, {($t \geq
2$)}, and  let  $\lambda_t$ {{be its} least   eigenvalue}. The
matrix  $A_t-\lambda_t I $ is doubly-nonnegative, its support is
$C_{2t+1}$, and it does not  admit {an}
$\mathcal{N}^+$-factorization for any tracial  von Neumann
algebra~$(\mathcal{N},\tau)$.
\end{lemma}
\medskip

\begin{proof}
Set $n:=2t+1$ and $X:=A_t-\lambda_t I$. Clearly, $X\in\DNN^n$ and $S(X)=C_{2t+1}$. Note that  $\lambda_t=2\cos({2\pi t\over 2t+1})$ with  {multiplicity $2$}. In particular  $\rank(X)=n-2.$   Let $X=\gram(\{u_i\}_{i=1}^n)$ {where $\{u_i\}_{i=1}^n \subseteq \R^{n-2}$} and {$\Span({\{ u_i\}_{i=1}^n})=\R^{n-2}$}.  We show that the assumptions  of Theorem \ref{thm:dnnminuscpsd} are
satisfied for  $I:=[n] \setminus \{ 2, n\}$ and $J:=[n] \setminus \{1,3\}$.
For $(i)$ note that {$\dim(\Span(\{u_i: i \in
I\}))=\dim(\Span(\{u_j: j \in J\}))=n-2$} and since $\{u_i\}_{i=1}^n \subseteq \R^{n-2}$,
we have $\Span(\{u_i: i \in
I\}))=\dim(\Span(\{u_j: j \in J\})$.
 Moreover, setting    $i^*:=1$
and $j^*: = 2, $  we see that $(ii)$ and $(iii)$ are satisfied. For
$(iv)$ note that  $\det \left(X[1,2]\right)=\lambda_t^2-1\ne 0$,
{where $X[1,2]$ denotes  the principal  submatrix of $X$ corresponding  to  the first two
rows and columns}. Lastly, $(v)$ holds as $\la u_1, u_2\ra=~{+1}$.
\end{proof}
\medskip

\begin{remark}\label{rem:oddcycles}
It was shown  in \cite{BLP}   that there exists a tracial  von
Neumann algebra $(\mathcal{N},\tau)$ such that  any element in the
closure of \  $\CPSD$ admits {an} $\mathcal{N}^+$-factorization.
Consequently, the matrices $A_t-\lambda_t I $ constructed in Lemma
\ref{lem:oddcycles}  are doubly-nonnegative and do not belong to the
closure of $\CPSD$.
\end{remark}

\subsection{Characterizing cpsd-graphs}\label{sec:whatever}

Using the family of matrices constructed in Lemma~\ref{lem:oddcycles} we are   now  ready to complete  our characterization of cpsd-graphs.

\medskip
\begin{theorem}\label{thm:cpsdgraphs}
A graph is cpsd  if and only if  it has no $C_{2t+1}$-subgraph $(t\ge 2)$.
\end{theorem}
\medskip

\begin{proof}
Consider a graph $G$ and {suppose it} has no $C_{2t+1}$-subgraph  for all $t \ge2$. Then $G$ is a cp-graph and thus, also a cpsd-graph.  Conversely, consider a graph $G$ that contains a $C_{2t+1}$-subgraph, for some  $t\ge 2$. We  show that $G$ is not a cpsd-graph.  First, {suppose} that $G=C_{2t+1}$ for some $t\ge 2$. It follows from {Lemma \ref{lem:oddcycles} and} Remark~\ref{rem:oddcycles} that  odd cycles of length at least 5 are not {cpsd-graphs} so we are done. Next {suppose} that $G=([n],E)$ contains $C_{2t+1}$ (for some $t\ge 2$) as a  proper subgraph. {Let $A_t$ and $\lambda_t$ be as  in Lemma \ref{lem:oddcycles}. Recall that {$X=A_t-\lambda_t I \in \DNN\setminus {\rm cl}(\CPSD)$}}. Let $\tilde{X}$ be the $n\times n$ matrix {whose principal submatrix corresponding to the vertices of $C_{2t+1}$  is given by $X$}, and all other entries are equal to $0$.
 For any $a>0$, since $\tilde{X}+aI$ is positive definite, we can find $0<b<a$ such that $X_a:=\tilde{X}+aI+bA_G\in \DNN$, {where $A_G$ is the adjacency
matrix of $G$}. %By Remark \ref{rem:oddcycles} we have
{By a continuity argument we see  that ${\rm cl}(\CPSD)$ is closed under taking principal
submatrices}.  {Thus}, as  $\lim_{a\rightarrow 0}X_a=\tilde{X}$ and
$\tilde{X}\not \in {\rm cl}(\CPSD)$, there exists  $a^*>0$ such that
$X_{a^*}\in \DNN\setminus  {\rm cl}(\CPSD)$.  In particular, we have
that $ X_{a^*}\in \DNN\setminus  \CPSD$. As $S(X_{a^*})=G$, it
follows that  $G$ is not a cpsd-graph. 
\end{proof}

\section*{Acknowledgments}

{We} thank Hamza Fawzi for bringing to our attention  reference  \cite{FW}.
A.V., A.P., and Z.W. are supported in part by the
Singapore National Research Foundation under NRF RF Award
No.~NRF-NRFF2013-13. {J.S. is supported in part by NSERC Canada. Research at the Centre for Quantum Technologies at the National University of Singapore is partially funded by the Singapore Ministry of Education
and the National Research Foundation, also through the Tier 3 Grant ``Random numbers from
quantum processes,'' (MOE2012-T3-1-009).}

\bibliographystyle{siamplain}
\bibliography{biblio}

\appendix

\section{Clifford algebras}\label{sec:Clifford}
Our goal in  this section is to briefly introduce  Clifford algebras. For additional details the reader is referred to  \cite[Chapter 6]{GW}.

Consider a  real vector space $V$ equipped with a    bilinear form $\beta: V\times V\rightarrow \R$ such that $(i)$ $\beta(x,y)=\beta(y,x), \forall x,y\in V$ and $(ii)$  $\beta$ is {\em non-degenerate}, i.e., $\forall x\in V, \beta(x,y)=0\Longrightarrow y=0
$.  A {\em Clifford algebra} for $(V,\beta )$ consists of  a real  unital associative  algebra denoted  ${\rm Cl}(V,\beta)$  together with a linear map $e: V \rightarrow {\rm Cl}(V,\beta)$   satisfying:
\vspace{0.1cm}

\bi
\item[$(i)$] ${e(u)e(v)+e(v)e(u)}=\beta( u,v)1, $ for all  $u,v\in V$;
\vspace{0.1cm}
\item[$(ii)$] ${\rm Cl}(V,\beta)$ is generated by {$e(V)$} as an algebra;
\vspace{0.1cm}
\item[$(iii)$] Given a    real unital associative  algebra $\mathcal{A}$ and  linear map $f: V \rightarrow \mathcal{A}$ satisfying $$f(u)f(v)+f(v)f(u)=\beta(u,v)
 1, \text{ for all } u,v\in V,$$ there exists a unique algebra homomorphism $h: {\rm Cl}(V,\beta) \rightarrow \mathcal{A} $ where~$f=h\circ e$.
\ei
\vspace{0.1cm}

A Clifford algebra  for $(V,\beta)$  can be explicitly defined as the quotient algebra $\mathcal{T}(V)/\mathcal{I}(V),$ where $\mathcal{T}(V):=\oplus_{k\ge 0 }V^{\otimes k}$ is the tensor algebra over $V$ and $\mathcal{I}(V)$ is the two-sided ideal in $\mathcal{T}(V)$  generated by the elements of the form  $u\otimes v+v\otimes u-\beta(u,v)1,$ for all $u,v\in V$. %By definition, the product of two elements $x,y\in {\rm Cl}(V)$ is given by $x\cdot y:=x\otimes y\mod \mathcal{I}(V)$, and we refer to this as the {\em Clifford product}.
Any two algebras satisfying conditions $(i), (ii), (iii)$ above are isomorphic. Thus we refer to ${\rm Cl}(V,\beta)$  as {\em the} Clifford algebra over~$V$.

A {\em representation} of an associative algebra $\mathcal{A}$
consists of a vector space $W$ together with an algebra homomorphism
$\Gamma: \mathcal{A} \rightarrow {\rm End}(W)$, i.e., a linear map
preserving multiplication and the unit element, {where ${\rm
End}(W)$ is the set of all endomorphisms of $W$}. The {\em
dimension} of a representation $(\Gamma,W)$ is the dimension of $W$
as a vector space.   A {\em subrepresentation} of a representation
$(\Gamma,W)$ is a subspace $U\subseteq W$ such that
$\Gamma(a)(U)\subseteq U$, for all $a\in \mathcal{A}$.    A
representation is called {\em irreducible} if its only
subrepresentations are itself and the trivial vector space.
 
It is well-known that the irreducible representations of ${\rm Cl}(V,\beta)$ have exponential size in terms of the dimension of $V$. This  is the source of {our exponential lower bound} in this paper.  Specifically, it is known that:

\medskip
\begin{theorem}\label{thm:cliffrepresentations}
Let $\beta$ be a nondegenerate bilinear form on $V$.
\bi
\item[$(i)$] If  $\dim V=2\ell $ then  (up to isomorphism) there exists a unique irreducible representation of ${\rm Cl}(V,\beta)$ {which} has dimension $2^\ell$;
\item[$(ii)$] If  $\dim V=2\ell +1$ then there exist two {nonisomorphic} irreducible representations of ${\rm Cl}(V,\beta)$. Both representations have dimension $2^\ell$.
\ei
\end{theorem}
\medskip

For a proof of this fact the reader is referred to \cite[Theorem 6.1.3]{GW}.

\medskip
\begin{remark}\label{rem:cscdd}Consider a  real vector space $V$ equipped with a   symmetric and non-degenerate  bilinear form $\beta: V\times V\rightarrow \R$.  Let $f$ be a linear map $f: V\rightarrow~{\rm End}(W)$ satisfying $f(u)f(v)+f(v)f(u)=\beta(u,v)1_W$, for all $u,v\in V$.
Using the three defining axioms for  ${\rm Cl}(V,\beta)$ it follows that $f$ can be extended to a representation for ${\rm Cl}(V,\beta)$.
\end{remark}

\section{Proof of Theorem \ref{thm:mainlowerbound1}}%Tsirelson's Theorem} 
\label{sec:dsvefwefwe}

In this section we give for completeness  a proof of Theorem~\ref{thm:mainlowerbound1}, as this is not stated explicitly in \cite{TS87}.  We start with a~definition.

\medskip
\begin{definition} Given $ C =(c_{xy}) \in \cornm$, we say that a family of  real vectors $\{u_x, v_y\}_{x,y}$ forms  a {\em $C$-system} if they satisfy
\[ \|u_x\|\le 1,\  \forall x,\  \|v_y\|\le 1,\ \forall y, \text{ and } c_{xy}=\la u_x,v_y\ra, \forall x,y. \]
\end{definition}
\medskip

{As it turns out, 
%In \cite{TS87} Tsirelson shows that 
$C$-systems of vectors corresponding to extremal quantum correlations have interesting properties. For our purposes  we only  need the following~result:}

 \medskip
  \begin{lemma}\cite[Lemma 3.1]{TS87}\label{lem:tsirelson}
Let  $ C\in {\rm ext}({\rm Cor}(n,m))$. Then, for any $C$-system of vectors $\{u_x,v_y\}_{x,y}$ we have that 
\be\label{deferfrfer}
  \Span(\{u_x\}_x) =  \Span(\{v_y\}_y).
\ee  
 Furthermore,  there exists an integer $\tau_C\ge 1$, depending only on $C$,   such that for any $C$-system of vectors $\{u_x,v_y\}_{x,y}$ we have that  
\be\label{new3} 
{\dim \left( \Span(\{u_x\}_x) \right)=\dim \left( \Span(\{v_y\}_y) \right)}=\tau_C. 
\ee
 Also, we can find $C$-systems $\{u_x,v_y\}_{x,y}$ that lie in $\R^{\tau_C}$ (and thus  span $\R^{\tau_C}$). For this,  let $\{a_x,b_y\}_{x,y}$ be an arbitrary  $C$-system and consider the matrix $\gram(\{a_x\}_x,\{b_y\}_y)$. By \eqref{new3} and \eqref{deferfrfer}, this  is a real psd matrix of  rank $\tau_C$ and thus any Gram factorization with vectors in $\R^{\tau_C}$ gives  a $C$-system with the required properties. Lastly, note that $\rank(C)\le~\tau_C$.
\end{lemma}
\medskip

We continue by stating  another result  due to Tsirelson which shows that  the operators in a  quantum representation of an extremal  quantum  correlation correspond   to a representation of an appropriate  Clifford algebra. This is the essential ingredient in the proof of Theorem \ref{thm:mainlowerbound1} given below.

\begin{theorem}\cite[Theorem 3.1]{TS87}\label{thm:tsirelson2} Let  $C=(c_{xy}) \in {\rm ext}({\rm Cor}(n,m))$
and consider a family of Hermitian operators $\{A_x\}_x,\{B_y\}_y,  \rho$ in $ \calh^d$  such that:
\bi
\item[$(i)$] $c_{xy}=\tr(A_xB_y\rho)$ for all $x,y$;
\item[$(ii)$] $A_xB_y=B_yA_x$;
\item[$(iii)$] $\rho$ is a density matrix;
\item[$(iv)$] The eigenvalues of $A_x,B_y$ are in $[-1,1]$;
\item[$(v)$] There does not exist an orthogonal  projector $P\ne I$ such
that 
\be\label{dcefe} PA_x=A_xP, \
PB_y=B_yP \text{ and } P\rho P=\rho. \ee \ei Then, for any
$C$-system of vectors {$\{ u_x \}_x$, $\{ v_y \}_{y}$} we have that
\be\label{eq:cliffordtsirelson} \{A_x,A_{x'}\}=2\la u_x,u_{x'}\ra
I_d,\  \forall x,x' \text{ and }\{B_{y},B_{y'}\}=2\la v_y,v_{y'}\ra
I_d,\  \forall y,y', \ee 
{where $\{A, B \} := AB + BA$ is the \emph{anticommutator} of $A$ and $B$.}
\end{theorem}
\medskip

Using Theorem \ref{thm:tsirelson2} we  are now ready to  give a proof for Theorem \ref{thm:mainlowerbound1}.

\medskip
\begin{thmnewww}\label{thm:lowerboundandproof}
Let $C=(c_{xy}) \in {{\rm ext}(\cornm)}$  and consider a  family of   Hermitian operators
$\{M_x\}_{x}, \{N_y\}_{y}\subseteq \calh^d$ with eigenvalues  in $[-1,1]$ and a quantum state $\rho\in~\calh^{d^2}_+$ satisfying
$c_{xy}=\tr((M_x\otimes N_y)\rho),$ for all $x,y$. Then we have that
$$d\ge {\sqrt{2}^{\lfloor\rank(C)/ 2\rfloor}}.$$
\end{thmnewww}
\medskip

\begin{proof}
For all $x$ set  $A_x:=M_x\otimes I_d \in \calh^{d^2}$ and  for all $y$ set  $B_y:=I_d\otimes N_y\in\calh^{d^2}$.
Note that conditions $(i)-(iv)$ of Theorem \ref{thm:tsirelson2}   are satisfied.  Furthermore, if there exists an orthogonal projector $P\ne I$  satisfying  \eqref{dcefe}, by restricting on the support of  the matrices  $\{ PA_{x} P\}_x,\{ PB_{y} P\}_y$ and $P\rho P$  we get a new family of operators that  satisfy conditions  $(i)-(iv)$ from Theorem \ref{thm:tsirelson2}
 that  have smaller size. This  process can be  repeated to obtain matrices  satisfying conditions $(i)-(v)$ whose size  is  at most $d^{2}$.

By Lemma \ref{lem:tsirelson} there exists a $C$-system of vectors $\{u_x,v_y\}_{x,y}$ satisfying
$$\Span(\{u_x\}_x)=\Span(\{v_y\}_y)=\R^{\tau_C}.$$
Furthermore, by Theorem \ref{thm:tsirelson2} we have
\be\label{dcdevfv}
\{A_x,A_{x'}\}=2\la u_x,u_{x'}\ra I_{d^2},\  \forall x,x' \text{ and }\{B_{y},B_{y'}\}=2\la v_y,v_{y'}\ra I_{d^2}, \ \forall y,y'.
\ee
To ease notation set $\tau:=\tau_C$ and without loss of generality  assume that $\{u_x\}_{x=1}^\tau$ is a basis for $\R^\tau$.
Set $f :\R^\tau \rightarrow \calh^{d^2}$ where $f(u_x):=A_x$, for $1\le x\le \tau$ and  extend linearly, i.e.,  $f(\lambda)=\sum_{x=1}^{\tau} \lambda_xA_x,$ for all $\lambda\in \R^\tau$, where  $\lambda=(\lambda_x)$ are the  coordinates with respect to the  $\{u_x\}_{x=1}^\tau$ basis. Using \eqref{dcdevfv} it follows that  for $\lambda,\mu\in \R^\tau$ we have
\be\label{cevrtgvr}
\{f(\lambda),f(\mu)\}=2 \lambda^\sfT \gram(\{u_x\}_{x=1}^\tau)\mu \cdot {I_{d^2}}.
\ee
Define the bilinear form $\beta:\R^\tau\times \R^\tau\rightarrow \R$ by
$$\beta(\lambda,\mu)=2 \lambda^\sfT\gram(\{u_x\}_{x=1}^\tau)\mu.$$
Note that $\beta$ is symmetric and  furthermore, since $\gram(\{u_x\}_{x=1}^\tau)$ is full-rank, $\beta$ is  also nondegenerate. By  \eqref{cevrtgvr},  the map $f$ can be extended to  a representation of
the Clifford algebra ${\rm Cl}(\R^\tau,\beta)$ (cf.   Remark \ref{rem:cscdd}). 
Any  representation  of ${\rm Cl}(\R^\tau,\beta)$ can be {decomposed} as
a direct sum of irreducible ones, which by Theorem
\ref{thm:cliffrepresentations} {have} size at least $2^{\lfloor
\tau/ 2\rfloor}$. This implies that $d^2 \ge 2^{\lfloor
\tau/2\rfloor} $ and thus $d\ge \sqrt{2}^{\lfloor \tau/2\rfloor}$.
Lastly, by Lemma~\ref{lem:tsirelson}  we have that $\tau\ge
\rank(C)$ and the proof is concluded. 
\end{proof}

\end{document}